\documentclass[reqno,12pt]{amsart}

\usepackage{amssymb}
\usepackage[mathscr]{eucal}
\usepackage[colorlinks]{hyperref}

\numberwithin{equation}{section}

\newcommand{\mc}{\mathcal}
\newcommand{\fing}{\g}
\newcommand{\finh}{\mf{h}}
\newcommand{\isomap}{{\;\stackrel{_\sim}{\to}\;}}
\newcommand{\gr}{\on{gr}}
\newcommand{\mf}{\mathfrak}
\newcommand{\Lam}{\Lambda}
\newcommand{\lam}{\lambda}
\newcommand{\cprime}{$'$}
\newcommand{\W}{\mathscr{W}}
\newcommand{\on}{\operatorname}
\newcommand{\1}{\mathbf {1}}
\newcommand{\Z}{{\mathbb Z}}
\newcommand{\Q}{{\mathbb Q}}
\newcommand{\C}{{\mathbb C}}
\newcommand{\M}{{\mathbb L}}
\newcommand{\g}{{\mathfrak g}}
\newcommand{\h}{{\mathfrak h}}
\newcommand{\wh}{{\widehat{\mathfrak h}}}
\newcommand{\wg}{{\widehat{\mathfrak g}}}
\newcommand{\I}{{\mathcal I}}

\newcommand{\J}{{\mathcal J}}

\newcommand{\CS}{{\mathcal S}}
\newcommand{\CT}{{\mathcal T}}
\newcommand{\CY}{{\mathcal Y}}

\newcommand{\mraff}{\mathrm{aff}}
\newcommand{\mrpara}{\mathrm{para}}

\newcommand{\al}{\alpha}

\newcommand{\dl}{\delta}
\newcommand{\gm}{\gamma}
\newcommand{\om}{\omega}

\newcommand{\la}{\langle}
\newcommand{\ra}{\rangle}
\newcommand{\bop}{\bigoplus}
\newcommand{\op}{\oplus}
\newcommand{\ot}{\otimes}

\newcommand{\hsl}{\widehat{\mf{sl}}_2}

\newcommand{\Vaff}{V^{\mathrm{aff}}}

\DeclareMathOperator{\Aut}{Aut}
\DeclareMathOperator{\ch}{ch}
\DeclareMathOperator{\End}{End}

\DeclareMathOperator{\wt}{wt}

\newtheorem{thm}{Theorem}[section]
\newtheorem{prop}[thm]{Proposition}
\newtheorem{lem}[thm]{Lemma}
\newtheorem{cor}[thm]{Corollary}
\newtheorem{rmk}[thm]{Remark}

\newtheorem{hyp}[thm]{Hypothesis}
\newtheorem{exm}[thm]{Example}
\begin{document}

\title
{Parafermion vertex operator algebras and W-algebras}

\author[T. Arakawa]{Tomoyuki Arakawa}
\address{Research Institute for Mathematical Sciences,
Kyoto University, Kyoto 606-8502, Japan}
\email{arakawa@kurims.kyoto-u.ac.jp}

\author[C. H. Lam]{Ching Hung Lam}
\address{Institute of Mathematics, Academia Sinica, Taipei 115, Taiwan}
\email{chlam@math.sinica.edu.tw}

\author[H. Yamada]{Hiromichi Yamada}
\address{Department of Mathematics, Hitotsubashi University, Kunitachi,
Tokyo 186-8601, Japan}
\email{yamada.h@r.hit-u.ac.jp}


\begin{abstract}
We prove the conjectual isomorphism between
the level $k$ $\hsl$-parafermion
vertex operator algebra and the $(k+1,k+2)$
minimal series $W_k$-algebra for all $k\in \mathbb{N}$.
As a consequence, we obtain the conjectural isomorphism 
between the $(k+1,k+2)$ minimal series $W_k$-algebra  
and the coset vertex operator algebra $SU(k)_1\otimes SU(k)_1/SU(k)_2$.
\end{abstract}

\maketitle

\section{Introduction}
A parafermion vertex operator algebra $K(\g,k)$ is by definition the
commutant of the Heisenberg vertex operator subalgebra $M_{\wh}(k,0)$ in
the simple vertex operator algebra $L_{\wg}(k,0)$ of the level $k$ integrable
highest weight module for an affine Kac-Moody algebra $\wg$, 
where $\g$ is a finite dimensional simple Lie algebra 
and $\mathfrak{h}$ is a Cartan subalgebra of $\mathfrak{g}$. 
Some basic properties of $K(\g,k)$ were studied in
\cite{DW1,DW2}. Their arguments heavily depend on the properties of the
parafermion vertex operator algebra $K(\mathfrak{sl}_2,k)$,
which were obtained in \cite{DLWY, DLY}. Thus,
for the study of the general parafermion vertex operator algebras,
it is essential  to 
understand the case $\g = \mathfrak{sl}_2$.

The parafermion vertex operator algebra $K(\mathfrak{sl}_2, k)$ is also known
as a $W$-algebra. 
It was conjectured over 20 years ago in the physics literature
\cite{BEHHH}  that the parafermion
vertex operator algebra $K(\mathfrak{sl}_2, k)$
is isomorphic to 
the $(k+1,k+2)$-minimal series $W$-algebra 
\cite{FKW, A2012Dec}
associated with $\mathfrak{sl}_k$. 
The purpose of this paper is to prove this conjecture.
As a consequence,
the rationality of $K(\mf{sl}_2,k)$ is established.

In \cite{DLWY, DLY} 
it was shown that $K(\mathfrak{sl}_2, k)$
is isomorphic to the simple quotient of the $W$-algebra
$W(2,3,4,5)$ of \cite{BEHHH, Hornfeck} for $k \ge 5$.
This relationship between $K(\mathfrak{sl}_2, k)$ and the simple quotient of
$W(2,3,4,5)$ is clear because a set of generators and the operator product
expansions among the generators of these two vertex operator algebras are
known \cite{BEHHH, DLWY, DLY}, and they coincide with each other.
On the other hand, the $W$-algebra $\W^\ell( \mathfrak{sl}_k)$ was constructed
in a different manner
and
the explicit operator product
expansions of generators of $\W^\ell( \mathfrak{sl}_k)$ are not known in general.
Therefore,
we should take another approach to establish a correspondence
between the vertex operator algebras $K(\mathfrak{sl}_2, k)$ 
and $\W^\ell( \mathfrak{sl}_k)$.

The key idea is the use of a decomposition
$L_{\hsl}(k,0) =  \bigoplus_{j=0}^{k-1} V_{\Z\gm - j\gm/k} \ot M^j$,
where $V_{\Z\gm - j\gm/k}$ is a simple module for
a vertex operator algebra $V_{\Z\gm}$ associated with a rank one
lattice $\Z\gm$, $\la \gm,\gm\ra = 2k$, and $M^j$ is a simple 
module for $M^0 = K(\mathfrak{sl}_2, k)$.
That is, we consider not only the vertex operator algebra
$K(\mathfrak{sl}_2, k)$ but also some of its simple modules and
take tensor product with $V_{\Z\gm - j\gm/k}$.
We shall characterize $L_{\hsl}(k,0)$ as a unique vertex operator algebra that admits such a decomposition,
see Section \ref{sec:characterization-Lk0}
for the precise statement.
The fusion rules among the simple $V_{\Z\gm}$-modules
play an important role in the argument here.
Furthermore, we shall apply
the result to show that under a certain assumption, 
a vertex operator algebra having similar properties as
the $(k+1,k+2)$-minimal series $W$-algebra $\W_{k+1,k+2}(\mathfrak{sl}_k)$
associated with $\mathfrak{sl}_k$ 
is in fact isomorphic to 
the parafermion vertex operator algebra $K(\mathfrak{sl}_2, k)$.

We note that 
by the level rank duality
$K(\mf{sl}_2,k)$ is isomorphic to the 
coset vertex algebra 
$\on{Com}_{L_{\widehat{\mf{sl}}_k}(1,0)\otimes L_{\widehat{\mf{sl}}_k(1,0)}}
(L_{\widehat{\mf{sl}}_k}(2,0))$
(\cite{Lam2}).
Therefore our result gives another conjectural isomorphism (\cite{KacWak90,BS})
\begin{align*}
\W_{k+1,k+2}(\mf{sl}_k)\cong 
\on{Com}_{L_{\widehat{\mf{sl}}_k}(1,0)\otimes L_{\widehat{\mf{sl}}_k(1,0)}}
(L_{\widehat{\mf{sl}}_k}(2,0))
\end{align*}
for all $k\geq 2$.

This paper is the final version of our unpublished preprint
``A characterization of parafermion vertex operator algebras''.

The organization of the paper is as follows.
In Section 2, we prepare some materials which will be necessary
in later sections. We review intertwining operators 
and simple current extensions.
We also recall some properties of the lattice vertex operator algebra
$V_{\Z\gm}$ and its simple modules as well as the construction
of the parafermion vertex operator algebra $K(\mathfrak{sl}_2, k)$.
In Section 3, we study a vertex operator algebra of the form
$V = \bigoplus_{j=0}^{k-1} V_{\Z\gm - j\gm/k} \ot M^j$.
Under a certain hypothesis, we show that $V$ is isomorphic to
$L_{\hsl}(k,0)$. 
In Section 4, we obtain a characterization of 
the parafermion vertex operator algebra $K(\mathfrak{sl}_2, k)$.
In Section 5, we recall some results
on W-algebras established in 
\cite{Arakawa, Ara10, A2012Dec,A-vEk}.
Finally, in Section 6, we apply the result of Section 4 
to show that $K(\mathfrak{sl}_2, k)$ is isomorphic to $\W_{k+1,k+2}(\mathfrak{sl}_k)$. 
A correspondence of the simple modules for $K(\mathfrak{sl}_2, k)$ with those for 
$\W_{k+1,k+2}(\mathfrak{sl}_k)$ is discussed as well. 
In Appendix A, we prove 
Proposition \ref{Pro:generator-of-W}
that was stated in 
\cite{FL88} without a proof.

\section{Preliminaries}
We use standard notation for vertex operator algebras and their modules
\cite{FHL, FLM, LL}. Let $(V,Y,\1,\om)$ be a vertex operator algebra and
$(M,Y_M)$ its module. Then 
\begin{align*}
Y_M(v,z) = \sum_{n \in \Z} v_n z^{-n-1}
\end{align*}
is the vertex operator associated with $v \in V$. The linear operator $v_n$
of $M$ is called a component operator. The eigenspace with eigenvalue
$n$ for the operator $L(0) = \om_1$ is denoted by $M_{(n)}$. An element
$w \in M_{(n)}$ is said to be of weight $n$ or $\om$-weight $n$ and
we write $\wt w = n$. In this paper we always assume that $\wt w \in \Q$.
The top level of $M$ means the nonzero weight subspace
$M_{(n)}$ of smallest possible $n$. We denote the top level of $M$ by $M(0)$. 
The weight of the top level $M(0)$ is called the top weight of $M$.  
The generating function
\begin{equation*}
\ch M = \sum_{n \in \Q} (\dim M_{(n)}) q^n
\end{equation*}
of $\dim M_{(n)}$ is called the character of $M$.
A vertex operator algebra $V$ is said to be of CFT-type
if $V = \bop_{n \ge 0} V_{(n)}$ and $V_{(0)} = \C\1$.

Let $M' = \bigoplus_{n \in \Q} M_{(n)}^\ast$ be the restricted dual space of
a $V$-module $M$, where $M_{(n)}^\ast$ is the ordinary dual space of $M_{(n)}$.
The adjoint vertex operator $Y_{M'}(v,z) \in (\End M')[[z,z^{-1}]]$ is
defined by
\begin{equation*}
\la Y_{M'}(v,z)w', w \ra_M =
\la w',Y_M(e^{zL(1)} (-z^{-2})^{L(0)} v,z^{-1})w \ra_M
\end{equation*}
for $v \in V$, $w \in M$ and $w' \in M'$, where $\la \,\cdot\,,\,\cdot\,\ra_M$ is
the natural pairing of $M'$ and $M$ \cite[(5.2.4)]{FHL}.
Then $(M',Y_{M'})$ is a $V$-module \cite[Theorem 5.2.1]{FHL} called the
contragredient or dual module of $M$. 
If $M$ and $M'$ are isomorphic as $V$-modules, then $M$ is said to be self-dual. 
The vertex operator algebra $V$ is said to be self-dual if $V$ is isomorphic to 
its dual $V'$ as a $V$-module.

For an automorphism $g$ of $V$, we define a $V$-module $(M \circ g, Y_{M \circ g})$ 
by setting $M \circ g = M$ as vector spaces and $Y_{M \circ g}(v,z) = Y_M(gv,z)$ for 
$v \in V$. Then $M \mapsto M \circ g$ induces a permutation on the set of 
simple $V$-modules. The $V$-module $M$ is said to be $g$-stable if 
$(M \circ g, Y_{M \circ g})$ is isomorphic to $(M,Y_M)$.

\subsection{Fusion rules}\label{subsec:fusion-rule}
We review intertwining operators for later use.
Let $V$ be a vertex operator algebra and $(U^i,Y_{U^i})$, $i = 1,2,3$ 
simple $V$-modules. 
Let $\CY(\,\cdot\,,z)$ be an
intertwining operator of type $\binom{U^3}{U^1\ U^2}$ \cite[Section 5.4]{FHL}.
We define linear operators $u^1_m$ and $u^1(n)$ from $U^2$ to $U^3$ by
\begin{align*}
\CY (u^1,z) u^2
&= \sum_{m \in \Q} u^1_m u^2 z^{-m-1}\\
&= \sum_{n \in \Z} u^1(n)u^2 z^{-n - h_1 - h_2 + h_3}
\end{align*}
for $u^i \in U^i$, $i=1,2$, $m \in \Q$ and $n \in \Z$,
where $h_i$ is the top weight of $U^i$. 
If $u^1$ is homogeneous, then the weight of the operator $u^1_m$ is
$\wt u^1_m = \wt u^1 - m - 1$ \cite[(5.4.14)]{FHL}.

In the case $U^1 = V$ and $u^1 = \1$, it is well-known that 
$\CY(\1,z) :  U^2 \rightarrow U^3$ is a homomorphism of $V$-modules 
(cf. page 285 of \cite{Li1}). 

Let $I \binom{U^3}{U^1\ U^2} = I_V \binom{U^3}{U^1\ U^2}$ be the 
space of intertwining operators of type $\binom{U^3}{U^1\ U^2}$. 

\begin{lem}\label{lem:nonzero-action}
Let $V$ be a self-dual vertex operator algebra and $U^1$, $U^2$
simple $V$-modules with integral weight. 
Assume that $I \binom{V}{U^1\ U^2} \ne 0$. 
Then the following assertions hold.

$(1)$ $U^1$ is isomorphic to the dual module $(U^2)'$ of $U^2$. 

$(2)$ Let $0 \ne \CY(\,\cdot\,,z) \in I \binom{V}{U^1\ U^2}$. 
Then for any $0 \ne u^1 \in U^1(0)$, there exists $u^2 \in U^2(0)$ such that 
the weight $0$ coefficient $u^1_{2h -1} u^2$ of $\CY(u^1,z)u^2$ is 
nonzero: $0 \ne u^1_{2h -1} u^2 \in V_{(0)}$, 
where $h$ is the top weight of $U^1$.
\end{lem}

\begin{proof}
Since the weights of $U^1$ and $U^2$ are integral and $V$ is self-dual, 
there are one-to-one correspondences among the four spaces
\begin{equation*}
I \binom{V}{U^1\ U^2}, \quad I \binom{(U^2)'}{U^1\ V'}, \quad 
I \binom{(U^2)'}{V'\ U^1}, \quad I \binom{(U^2)'}{V\ U^1}
\end{equation*}
of intertwining operators by \cite[Propositions 5.4.7 and 5.5.2]{FHL}. 
More precisely, let $\psi: V \rightarrow V'$ be an isomorphism of  $V$-modules 
and $0 \ne \CY(\,\cdot\,,z) \in I \binom{V}{U^1\ U^2}$. 
We consider three intertwinig operators 
\begin{equation*}
\CY^1(\,\cdot\,,z) \in I \binom{(U^2)'}{U^1\ V'}, 
\quad \CY^2(\,\cdot\,,z) \in I \binom{(U^2)'}{V'\ U^1}, 
\quad \CY^3(\,\cdot\,,z) \in  I \binom{(U^2)'}{V\ U^1}
\end{equation*}
defined by 
\begin{align}
\la \CY^1(u^1,z) \psi(a), u^2\ra_{U^2} 
&= \la \psi(a), \CY(e^{zL(1)}(-z^{-2})^{L(0)} u^1, z^{-1}) u^2\ra_V, \label{eq:Y1}\\
\CY^2(\psi(a),z)u^1 &= e^{zL(-1)}\CY^1(u^1,-z)\psi(a), \label{eq:Y2}\\
\CY^3(a,z)u^1 &= \CY^2(\psi(a),z)u^1, \label{eq:Y3}
\end{align}
respectively for $a \in V$, $u^1 \in U^1$ and $u^2 \in U^2$ 
\cite[(5.5.4), (5.4.33)]{FHL}.

Let $a = \1$ and $0 \ne u^1 \in U^1(0)$. 
Then $\CY^3(\1,z) \ne 0$ by \cite[Proposition 11.9]{DL}. 
Since $U^1$ and $(U^2)'$ are simple by our assumption and 
\cite[Proposition 5.3.2]{FHL}, $\CY^3(\1,z) : U^1 \rightarrow (U^2)'$ is 
in fact an isomorphism of  $V$-modules. 
Thus the assertion (1) holds.

We can choose $u^2 \in U^2(0)$ so that
\begin{equation}\label{eq:Y3-nonzero}
\la \CY^3(\1,z)u^1, u^2\ra_{U^2} \ne 0, 
\end{equation}
for $\CY^3(\1,z)u^1$ is a nonzero element of the dual space 
$U^2(0)^\ast$ of  $U^2(0)$. 
The weight of $u^2$ coincides with that of $u^1$. 
It follows from \eqref{eq:Y2}, \eqref{eq:Y3} and \eqref{eq:Y3-nonzero} that 
\begin{equation}\label{eq:Y1-nonzero}
\la e^{zL(-1)} \CY^1(u^1,-z)\psi(\1), u^2\ra_{U^2} \ne 0.
\end{equation}

Since $u^2 \in U^2(0)$, we have $L(1)u^2 = 0$ and  
\begin{equation}\label{eq:Y1-nonzero-2}
\la \CY^1(u^1,-z)\psi(\1), u^2\ra_{U^2} \ne 0
\end{equation}
by \eqref{eq:Y1-nonzero} and \cite[(5.2.10)]{FHL}. 
Recall that $u^1 \in U^1(0)$ and the weight of $u^1$ is $h$. Then 
\begin{equation*}
e^{-zL(1)} (-z^{-2})^{L(0)} u^1 = (-z^{-2})^h u^1
\end{equation*}
and so 
\begin{equation*}
\la \CY^1(u^1,-z) \psi(\1), u^2\ra_{U^2} 
= (-z^{-2})^h \la \psi(\1), \CY(u^1, -z^{-1}) u^2\ra_V,
\end{equation*}
which is nonzero by \eqref{eq:Y1-nonzero-2}. 
Since $\psi(\1)$ is an element of the dual space $V_{(0)}^\ast$ of $V_{(0)}$, 
we conclude that $u^1_{2h -1} u^2$ is a nonzero element of $V_{(0)}$.
\end{proof}

\subsection{Simple current extensions}
A simple module $M$ of a vertex operator algebra is called a simple current 
if the tensor product $M \boxtimes N$ exists and it is a simple module for 
every simple module $N$.
We review some known results about simple current extensions of 
vertex operator algebras for later use.

We assume the following hypothesis.

\begin{hyp}\label{hypothesis:0}
$(1)$ $V$ is a simple, self-dual, rational and $C_2$-cofinite vertex operator algebra of CFT-type. 

$(2)$ The top weight of any simple $V$-module $M$ is non-negative,
and is zero only if $M=V$.

$(3)$ $U^i$, $i \in D$ is a set of 
simple current $V$-modules with integral weight, where $D$ is a finite abelian group and $U^0 = V$. 
The fusion rules among $U^i$'s are 
\begin{equation*}
U^i \times U^j = U^{i+j}, \quad i, j \in D.
\end{equation*}
\end{hyp}


%

Let $\I_{i,j}^{i+j} = I_V \binom{U^{i+j}}{U^i\ U^j}$ be the space of intertwining 
operators of type $\binom{U^{i+j}}{U^i\ U^j}$ for $i, j \in D$. 
Let $\CY_{i,j}^{i+j}$ be a nonzero element of $\I_{i,j}^{i+j}$. 
Since $\I_{i,j}^{i+j}$ is one dimensional, $\CY_{i,j}^{i+j}$ is unique up to 
a nonzero scalar multiple. 
However it is far from trivial whether there are $\lambda_{i,j} \in \C^\times$ such that 
$\{\lambda_{i,j}\CY_{i,j}^{i+j} \}_{i,j \in D}$ gives a vertex operator algebra structure 
on a direct sum $\bigoplus_{i \in D} U^i$. 

For $i,j \in D$, define $\Omega(i,j) \in \C^\times$ by 
\begin{equation}\label{eq:def-Omega}
\CY_{i,j}^{i+j}(u,z)v = \Omega(i,j) e^{z L(-1)} \CY_{j,i}^{i+j}(v,-z)u
\end{equation}
for $u \in U^i$, $v \in U^j$ \cite[Definition 2.2.4]{Carnahan} 
(see also \cite{CKL, vEMS}). 
Such a constant $\Omega(i,j)$ exists by \cite[Proposition 5.4.7]{FHL}.

The following fact has been shown in the proof of \cite[Theorem 4.1]{vEMS} 
(see also \cite[Theorem 3.12]{CKL}).
\begin{prop}
$\Omega(i,i)=1$, $i \in D$.
\end{prop}

Indeed, the categorical dimension of $U^i$ coincides with the quantum dimension 
in the sense of \cite[Definition 3.1]{DJX} by \cite[Eq. (4.1)]{DJX} 
under Hypothesis \ref{hypothesis:0}. 
Since $U^i$ is a simple current, the quantum dimension of $U^i$ is $1$ 
\cite[Lemma 4.15]{DJX}. 
Then the condition that $U^i$ has integral weight implies $\Omega(i,i)=1$ 
by the equation $e(q_\Omega(\alpha)) = e(-q_\Delta(\alpha))$ in the proof of 
\cite[Theorem 4.1]{vEMS}, for $q_\Delta(\alpha) = 0$ with $\alpha = i$.



The condition that $\Omega(i,i)=1$, $i \in D$ is called evenness in \cite{Carnahan}. 
The evenness implies the next theorem 
\cite[Theorem 3.2.12]{Carnahan}, \cite[Theorem 3.12]{CKL}, \cite[Theorem 4.2]{vEMS}.

\begin{thm}\label{thm:SCE}
There exists a choice of nonzero 
intertwining operators  $\CY_{i,j}^{i+j} \in \I_{i,j}^{i+j}$, $i,j \in D$ 
which gives a vertex operator algebra structure on $\bigoplus_{i \in D} U^i$ 
as an extension of $V$. 
\end{thm}



Such a vertex operator algebra structure on $\bigoplus_{i \in D} U^i$ is unique 
up to isomorphism \cite[Proposition 5.3]{DM2}. 
The vertex operator algebra $\bigoplus_{i \in D} U^i$ 
is called a simple current extension of $V$. 
It is a simple, self-dual, rational and $C_2$-cofinite vertex operator algebra of CFT-type 
\cite[Theorem 2.14]{Yamauchi}.


\subsection{Lattice vertex operator algebra $V_{\Z\gm}$}\label{subsec:latticeVOA}
We recall some basic properties of a vertex operator algebra
associated with a positive definite even rank one lattice. 
Let $\Z\gm$ be a positive definite even rank one lattice generated by $\gm$, where 
the square norm of $\gm$ is $\la \gm, \gm \ra = 2k$. 
Let $V_{\Z\gm} = M(1) \ot \C[\Z\gm]$ be a vertex operator algebra
associated with the lattice $\Z\gm$ \cite{FLM}.
Thus $M(1)$ is 
a simple highest weight module for the
Heisenberg algebra generated by $\gm$ with highest weight $0$. It is
isomorphic to a polynomial algebra $\C[ \gm(-n)\,|\, n \in \Z_{> 0}]$ as
a vector space. Since $\Z\gm$ is a rank one lattice, the twisted
group algebra $\C\{\Z\gm\}$ considered in \cite{FLM} is isomorphic to
an ordinary group algebra $\C[\Z\gm]$. Its standard basis is
$\{ e^{n\gm}\,|\, n \in \Z\}$ with multiplication
$e^\al e^\beta = e^{\al + \beta}$. The conformal vector is
\begin{equation}\label{eq:Vir-lattice}
\om_\gm = \frac{1}{4k} \gm(-1)^2\1
\end{equation}
and its central charge is $1$.

The vertex operator algebra $V_{\Z\gm}$ is
simple, self-dual, rational,  $C_2$-cofinite and of CFT-type.
The simple modules for  $V_{\Z\gm}$ were classified \cite{Dong}.
Since the dual lattice of $\Z\gm$ is $(1/2k)\Z\gm$, any simple
$V_{\Z\gm}$-module is isomorphic to one of $V_{\Z\gm + i\gm/2k}$,
$0 \le i \le 2k-1$. 
The top level of $V_{\Z\gm + i\gm/2k}$ is $\C e^{i\gm/2k}$ with weight
$i^2/4k$ if $0 \le i < k$ and $\C e^{(2k-i)\gm/2k}$ with weight
$(2k-i)^2/4k$ if $k < i \le 2k-1$. 
In the case $i = k$, the top level is $\C e^{\gm/2} + \C e^{-\gm/2}$ with
weight $k/4$.

The fusion rules among these simple
modules are also known and intertwining operators were
constructed by using vertex operators \cite[Chapter 12]{DL}.
In fact, the fusion rules are
\begin{equation}\label{eq:VZgm-fusion}
V_{\Z\gm + i\gm/2k} \times V_{\Z\gm + j\gm/2k} = V_{\Z\gm + (i+j)\gm/2k}.
\end{equation}
In particular,
all the simple modules are simple currents.

Since the commutator map $(\,\cdot\,,\,\cdot\,)$ of \cite[(12.5)]{DL} 
is trivial for the rank one lattice $(1/k)\Z\gm$
and since $\la\gm,i\gm/k\ra = 2i \in 2\Z$ implies
$(-1)^{\la \overline{a},\overline{b}\ra} = 1$ in \cite[(12.5)]{DL}, 
the vertex operator $Y(\,\cdot\,,z)$ on
$V_{(1/k)\Z\gm} = M(1) \otimes \C[(1/k)\Z\gm]$ defined in 
\cite[Section 8.4]{FLM} (see also \cite[Chapter 3]{DL}) itself can be
taken as the intertwining operator $\CY_{\lambda_j}(\,\cdot\,,z)$
of \cite[(12.3)]{DL} for $V_{\Z\gm + j\gm/k}$, $0 \le j \le k-1$.
That is, $(-1)^{\la \overline{a},\overline{b}\ra} c(\overline{a},\overline{b}) = 1$ 
in \cite[(12.5)]{DL} and so $Y(v,z)w$ with $v \in V_{\Z\gm + i\gm/k}$
and $w \in V_{\Z\gm + j\gm/k}$ satisfies the Jacobi identity for
intertwining operators \cite[(12.8)]{DL}.

Let $\CY_{L,i,j}^{i+j}(\,\cdot\,,z)$ be the vertex operator $Y(\,\cdot\,,z)$ 
for $V_{(1/k)\Z\gm}$ restricted to $V_{\Z\gm - i \gm/k}$ and 
acting on $V_{\Z\gm - j \gm/k}$, so that $\CY_{L,i,j}^{i+j}(u,z)v = Y(u,z)v$ 
for $u \in V_{\Z\gm - i \gm/k}$ and $v \in V_{\Z\gm - j \gm/k}$. 
It is an intertwining operator of type 
\begin{equation*}
\begin{pmatrix}
V_{\Z\gm - (i+j) \gm/k}\\
V_{\Z\gm - i \gm/k} \quad V_{\Z\gm - j \gm/k}
\end{pmatrix}
\end{equation*}
for the vertex operator algebra $V_{\Z\gm}$. 
The action of $Y(e^{\pm \gm/k},z)$ on the top level of
$V_{\Z\gm \pm \gm/k}$ will be used later. 
By the definition
\begin{equation*}
Y(e^{\pm \gm/k},z) = E^-(\mp \gm/k,z)E^+(\mp \gm/k,z) e^{\pm \gm/k} z^{\pm \gm/k},
\end{equation*}
where $E^\pm(\al,z) = \exp(\sum_{n \in \pm\Z_{> 0}} \frac{\al(n)}{n} z^{-n})$.
Since $e^{\gm/k} z^{\gm/k} e^{-\gm/k} = \1 z^{-2/k}$, we have
\begin{equation}\label{eq:intertwiner-lattice1}
Y(e^{\gm/k},z) e^{-\gm/k} = \1  z^{-2/k} + \frac{1}{k} \gm(-1) \1  z^{1-2/k} + \cdots.
\end{equation}

\subsection{Parafermion vertex operator algebra $K(\mathfrak{sl}_2,k)$}\label{subsec:parafermion}
We recall from \cite{ALY, DLWY, DLY} the properties of
parafermion vertex operator algebra associated with $\mathfrak{sl}_2$. 
Let $k \ge 3$ be an integer.
Let $\{ h,e,f \}$ be a standard Chevalley basis of the Lie algebra $\mathfrak{sl}_2$,
so that $[h,e] = 2e$, $[h,f] = -2f$, $[e,f] = h$ for the bracket
and $(h|h) = 2$, $(e|f) = 1$,
$(h|e) = (h|f) = (e|e) = (f|f) = 0$ for
the normalized invariant inner product.

Let $V(k,0) = V_{\hsl}(k,0)$ be 
a Weyl module for 
the affine Lie algebra $\hsl = \mathfrak{sl}_2 \ot \C[t,t^{-1}] \op \C C$ at level $k$. 
Denote by $\1$ its canonical highest weight vector,
which is called the vacuum vector. Then $\mathfrak{sl}_2 \ot \C[t]$
acts as $0$  and $C$ acts as $k$ on $\1$, and $V(k,0)$ is the
induced module of the $\mathfrak{sl}_2 \ot \C[t] \op \C C$-module
$\C\1$. We write $a(n)$ for the action of $a \ot t^n$ on $V(k,0)$.
The Weyl module $V(k,0)$ is a vertex operator algebra
with the conformal vector
\begin{equation*}
\om_{\mraff} = \frac{1}{2(k+2)}
\Big( \frac{1}{2}h(-1)^2\1 + e(-1)f(-1)\1 + f(-1)e(-1)\1 \Big),
\end{equation*}
whose central charge is $3k/(k+2)$ \cite{FZ}, \cite[Section 6.2]{LL}.

Let $M_{\wh}(k,0)$ be the vertex operator subalgebra of $V(k,0)$
generated by $h(-1)\1$. That is, $M_{\wh}(k,0)$ is a Heisenberg
vertex operator algebra.
The conformal vector of $M_{\wh}(k,0)$ is
\begin{equation*}
\om_h = \frac{1}{4k} h(-1)^2\1
\end{equation*}
and its central charge is $1$.
As a module for $M_{\wh}(k,0)$, we have a decomposition
\begin{equation*}
V(k,0) = \bigoplus_{\lambda \in 2\Z}
M_{\wh}(k,\lambda) \ot N_\lambda,
\end{equation*}
where $M_{\wh}(k,\lambda)$ is a simple highest weight
module for $M_{\wh}(k,0)$ with a highest weight vector
$v_\lambda$ such that $h(0) v_\lambda = \lambda v_\lambda$ and
\begin{equation}\label{eq:Nlambda}
N_\lambda = \{ v \in V(k,0)\,|\, h(m)v = \lambda \dl_{m,0} v
\text{ for } m \ge 0 \}.
\end{equation}

In particular, $N_0$ is the commutant \cite[Theorem 5.1]{FZ}
of $M_{\wh}(k,0)$ in $V(k,0)$, which is a vertex operator
algebra with the conformal vector $\om_{\mrpara} = \om_{\mraff} - \om_h$.
The central charge of $N_0$ is $2(k-1)/(k+2)$.
The character of $N_0$ is $\ch N_0 = 1 + q^2 + 2q^3 + \cdots$.
It is known \cite[Section 2]{DLY} that
\begin{equation}\label{eq:W3}
\begin{split}
W^3 &= k^2 h(-3)\1 + 3 k h(-2)h(-1)\1 +
2h(-1)^3\1 - 6k h(-1)e(-1)f(-1)\1 \\
& \quad + 3 k^2e(-2)f(-1)\1 - 3 k^2e(-1)f(-2)\1.
\end{split}
\end{equation}
is a unique, up to a scalar multiple, Virasoro singular
vector in the weight $3$ subspace $(N_0)_{(3)}$.
The vertex operator algebra $N_0$ is generated by the
conformal vector $\om_{\mrpara}$ and the weight $3$
vector $W^3$ \cite[Theorem 3.1]{DLWY}.

The vertex operator algebra $V(k,0)$ has a unique maximal
ideal $\J$, which is generated by a single element
$e(-1)^{k+1}\1$ \cite{K}.
Let $L(k,0) = L_{\hsl}(k,0) = V(k,0)/\J$.
Since $M_{\wh}(k,0) \cap \J = 0$, $M_{\wh}(k,0)$ can be
considered as a subalgebra of $L(k,0)$ and we have
a decomposition
\begin{equation*}
L(k,0) = \bigoplus_{\lambda \in 2\Z}
M_{\wh}(k,\lambda) \ot K_\lambda
\end{equation*}
of $M_{\wh}(k,0)$-modules, where
\begin{equation*}
K_\lambda = \{ v \in L(k,0)\,|\, h(m)v = \lambda \dl_{m,0} v
\text{ for } m \ge 0 \}.
\end{equation*}

Note that $K_0$ is the commutant of $M_{\wh}(k,0)$ in $L(k,0)$. 
We use the same symbols $a(-1)\1$ for $a \in \{ h,e,f \}$,
$\om_{\mraff}$, $\om_h$, $\om_{\mrpara}$ and $W^3$ to denote their images
in $L(k,0)$.

We call $K_0$ a parafermion vertex operator algebra
associated with $\mathfrak{sl}_2$ and denote it by $K(\mathfrak{sl}_2,k)$. 
It is a simple vertex operator algebra of
central charge $2(k-1)/(k+2)$ and generated by $\om_{\mrpara}$ and $W^3$.
The character is $\ch K_0 = 1 + q^2 + 2q^3 + \cdots$.

The vertex operator algebra $K_0$ can be embedded in a
vertex operator algebra $V_L$ associated with a rank $k$
lattice $L = \Z\al_1 + \cdots + \Z\al_k$ with
$\la \al_p,\al_q \ra = 2\dl_{p,q}$
\cite[Section 4]{DLY}, \cite[Chapter 14]{DL}.
In fact, let $\gm = \al_1 + \cdots + \al_k$ and set
\begin{equation*}
H = \gm(-1)\1, \qquad E = e^{\al_1} + \cdots + e^{\al_k},
\qquad F = e^{-\al_1} + \cdots + e^{-\al_k}.
\end{equation*}

Then $\la \gm,\gm \ra = 2k$ and 
the component operators $H_n$, $E_n$, $F_n$, $n \in \Z$
give a level $k$ representation of $\hsl$ under the
correspondence
$h(n) \leftrightarrow H_n$,
$e(n) \leftrightarrow E_n$,
$f(n) \leftrightarrow F_n$.
In particular, the vertex operator subalgebra $\Vaff$
of $V_L$ generated by $H$, $E$ and $F$ is isomorphic to
$L(k,0)$. We also consider the vertex operator
subalgebra $V^\gm$ of $V_L$ generated by $e^\gm$ and
$e^{-\gm}$. Note that $V^\gm \cong V_{\Z\gm}$.
We identify $\Vaff$ with $L(k,0)$ and $V^\gm$ with
$V_{\Z\gm}$. We also identify $H_n$ with $h(n)$,
$E_n$ with $e(n)$ and $F_n$ with $f(n)$.
Then we have \cite[Lemma 4.2]{DLY}
\begin{equation}\label{eq:Lk0-dec}
L(k,0) = \bigoplus_{j=0}^{k-1}
V_{\Z\gm - j\gm/k} \ot M^j
\end{equation}
as $V_{\Z\gm} \otimes M^0$-modules, where
\begin{equation}\label{def:Mj}
M^j = \{ v \in L(k,0) | \gm(m)v = -2j \dl_{m,0} v
\text{ for } m \ge 0 \}.
\end{equation}
That is, $M^j = K_{-2j}$ for $0 \le j \le k-1$.
In particular,
$$M^0=K_0=K(\mathfrak{sl}_2,k).$$

\begin{rmk}
$M^j$ is denoted by $M^{0,j}$ in \cite[Lemma 4.2]{DLY}. 
The index $j$ of $M^j$ is considered to be modulo $k$.
\end{rmk}

Those $M^j$'s are simple $M^0$-modules \cite[Theorem 4.4]{DLY}. 
Hence \eqref{eq:Lk0-dec}
is a decomposition of $L(k,0)$ as a direct sum of
simple $V_{\Z\gm} \ot M^0$-modules
$V_{\Z\gm - j\gm/k} \ot M^j$, $0 \le j \le k-1$.

Let $L^\perp = \frac{1}{2}L$ be the dual lattice of $L$. 
The simple module $L(k,i) = L_{\hsl}(k,i)$ for the simple affine 
vertex operator algebra $L(k,0)$ with $i+1$ 
dimensional top level of weight $\frac{i(i+2)}{4(k+2)}$ can be constructed in 
the simple $V_L$-module $V_{L^\perp}$, $1 \le i \le k$. 
Let
\begin{equation}\label{def:Mij}
M^{i,j} = \{ v \in L(k,i) | \gm(m) v = (i-2j)\delta_{m,0}v \text{
for } m\ge 0\}
\end{equation}
for $0 \le i \le k$, $0 \le j \le k - 1$. 
Then $M^{i,j}$'s are simple $M^0$-modules and
\begin{equation}\label{eq:Lki-dec}
L(k,i) = \bigoplus_{j=0}^{k-1} V_{\Z\gm + (i-2j)\gm/2k} \otimes M^{i,j}
\end{equation}
as $V_{\Z\gm} \otimes M^0$-modules \cite[Lemma 4.3]{DLY}.

\begin{rmk}\label{rmk:j-mod-k}
The decomposition \eqref{def:Mij} implies that $M^{i,j}$ is the multiplicity of 
$V_{\Z\gm + (i-2j)\gm/2k}$ in $L(k,i)$. 
Since for a fixed $i$, $0 \le i \le k$, the simple $V_{\Z\gm}$-module 
$V_{\Z\gm + (i-2j)\gm/2k}$ is determined by $j \pmod{k}$, 
we consider the second index $j$ of  $M^{i,j}$ to be modulo $k$.
\end{rmk}

The $-1$ isometry of the lattice $L$ lifts to an automorphism $\theta$ of 
the vertex operator algebra $V_L$ of order $2$. 
Actually, $\theta(H) = -H$, $\theta(E) = F$ and $\theta(F) = E$. 
Both $\Vaff = L(k,0)$ and $M^0 = K(\mathfrak{sl}_2,k)$ are invariant under $\theta$. 
In fact, the automorphism group $\Aut M^0$ of $M^0$ is generated by $\theta$ 
and we have $\theta(W^3) = -W^3$.

\medskip
We summarize the properties of $M^0$ and
$M^{i,j}$ for $0 \le i \le k$, $0 \le j \le k-1$ 
(see \cite[Theorem 4.4, Proposition 4.5]{ALY}, 
\cite[Theorem 4.1]{DLWY}, \cite[Theorem 8.2]{DLY}).

(1) $M^0$ is a simple vertex operator algebra of central charge
$2(k-1)/(k+2)$.

(2) $\ch M^0 = 1 + q^2 + 2q^3 + \cdots$.

(3) $M^0$ is generated by $M^0_{(2)}$ and $M^0_{(3)}$.

(4) The simple $M^0$-modules $M^{i,j}$'s are not always inequivalent. In fact, 
\begin{equation}\label{eq:equiv-Mij}
M^{i,j} \cong M^{k-i,j-i}
\end{equation}
for $0 \le i \le k$, $0 \le j \le k-1$. 

(5) $M^{i,j}$, $0 \le j < i \le k$ form a complete set of representatives of 
the isomorphism classes of simple $M^0$-modules. 
There are exactly $k(k+1)/2$ inequivalent simple $M^0$-modules. 

(6) The top level of $M^{i,j}$ is one dimensional and its weight is 
\begin{equation}\label{eq:top-weight-Mij}
\frac{1}{2k(k+2)}\Big( k(i-2j) - (i-2j)^2 + 2k(i-j+1)j \Big)
\end{equation}
for $0 \le j < i \le k$.

(7) The automorphism $\theta$ of $M^0$ induces a permutation 
\begin{equation}\label{eq:theta-action}
M^{i,j} \mapsto M^{i,j} \circ \theta = M^{i,i-j}
\end{equation}
on the simple $M^0$-modules for $0 \le i \le k$, $0 \le j \le k-1$.

\begin{rmk}
It follows from \eqref{eq:equiv-Mij} that 
$M^{k,j}$ is isomorphic to $M^{0,j-k} = M^{0,j}$ for $0 \le j \le k-1$, 
even though the simple $L(k,0)$-module $L(k,k)$ is not isomorphic to $L(k,0)$ 
in the decomposition \eqref{eq:Lki-dec}. 
We also note that $M^j$ is equal to $M^{0,j} = M^{k,j}$ and 
its top weight is $j(k-j)/k$ for $0 \le j \le k-1$ by \eqref{eq:top-weight-Mij}. 
The top weight \eqref{eq:top-weight-Mij} of $M^{i,j}$ is non-negative, and is zero 
only if $i=k$ and $j=0$.
\end{rmk}

\section{A characterization of $L_{\hsl}(k,0)$}\label{sec:characterization-Lk0}
Let $k \ge 3$ be an integer. 
In this section we argue that a simple vertex operator algebra satisfying the following hypothesis is
isomorphic to the affine vertex operator algebra $L_{\hsl}(k,0)$.

\begin{hyp}\label{hypothesis:1}
$(1)$ $(V,Y,\1,\om)$ is a simple vertex
operator algebra of CFT-type with central charge $3k/(k+2)$.

$(2)$ $V$ contains a vertex operator subalgebra $(T^0,Y,\1,\om^1)$ isomorphic to
$V_{\Z\gm}$, where $\la \gm,\gm\ra = 2k$.
We identify $T^0$ with $V_{\Z\gm}$.
Then $\om^1 = \frac{1}{4k}\gm(-1)^2\1$ is the conformal vector of $T^0$
with central charge $1$. We assume that $\om^1 \in V_{(2)}$ and $\om_2 \om^1 = 0$.

$(3)$ Let $N^0$ be the commutant of $T^0$ in $V$ and
set $\om^2 = \om - \om^1$. Thus $(N^0,Y,\1,\om^2)$ is a vertex operator
subalgebra of $V$ with central charge $2(k-1)/(k+2)$.
We assume that $\ch N^0 = 1 + q^2 + 2q^3 + \cdots$ and that
$N^0$ is generated by
$N^0_{(2)}$ and $N^0_{(3)}$ as a vertex operator algebra.

$(4)$ We assume that $V$ is isomorphic to
$\bigoplus_{j=0}^{k-1} V_{\Z\gm - j\gm/k} \otimes N^j$ as a $T^0$-module, where
\begin{equation*}
N^j = \{ v \in V\,|\,\gm(m)v = -2j\delta_{m,0}v \text{ for } m \ge 0\}.
\end{equation*}
We also assume that as a module for $N^0$, the top weight of $N^j$ is $j(k-j)/k$.
\end{hyp}

Under Hypothesis \ref{hypothesis:1} we shall show that $V$ is isomorphic to 
$L(k,0) = L_{\hsl}(k,0)$.
The proof is divided into several steps.
First, we shall introduce some notation.
From the hypothesis we may assume
that $T^0 \ot N^0$ is a vertex operator subalgebra of $V$
\cite[Theorem 5.1]{FZ}.
Then the vacuum
vector and the conformal vector of $V$ are given as $\1 = \1^1 \ot \1^2$
and $\om = \om^1 \ot \1^2 + \1^1 \ot \om^2$, where $\1^1$ and $\1^2$ are
the vacuum vectors of $T^0$ and $N^0$, respectively.
For simplicity, we usually do not distinguish between $T^0 \ot \1^2$ and $T^0$
(resp. $\1^1 \ot N^0$ and $N^0$) and so $\om^1 \ot \1^2$ and $\om^1$
(resp. $\1^1 \ot \om^2$ and $\om^2$).
The weight of $v \in V$ as a module for $T^0$ (resp. $N^0$) or
$\om^1$-weight (resp. $\om^2$-weight) 
means the eigenvalue for the operator $L^1(0)$ (resp $L^2(0)$),
where $L^i(n) = \om^i_{n+1}$.

By our hypothesis, $V$ decomposes into a direct sum of simple $T^0$-modules 
and each simple direct summand is isomorphic to
one of $V_{\Z\gm - j\gm/k}$, $0 \le j \le k-1$. 
Moreover, 
$N^j$ is the sum of top levels of all simple
$T^0$-submodules of $V$ isomorphic to $V_{\Z\gm - j\gm/k}$.
We examine the action of $\gm(0) = (\gm(-1)\1)_0$
on the top level of each direct summand.

Let $\sigma = \exp(2\pi\sqrt{-1}\gm(0)/2k)$, which is an automorphism of
the vertex operator algebra $V$ of order $k$.
We consider its eigenspace
\begin{equation}\label{def:Vj}
V^j = \{ v \in V\,|\, \sigma v = \exp(-2\pi j\sqrt{-1}/k) v\}
\end{equation}
with eigenvalue $\exp(-2\pi j\sqrt{-1}/k)$.
Then $V = \bigoplus_{j=0}^{k-1} V^j$.
By \cite[Theorem 3]{DM}, $V^0$ is a simple vertex operator algebra and
$V^j$, $1 \le j \le k-1$ are simple $V^0$-modules.
For convenience, we understand the index $j$ of $V^j$ to be modulo $k$.
Since $\sigma$ is an automorphism of $V$, we have
\begin{equation}\label{eq:unv}
u_n v \in V^{i+j} \quad \text{for} \quad u \in V^i, v \in V^j, n \in \Z.
\end{equation}

In fact, $V^j$ is the sum of all simple
$T^0$-submodules of $V$ isomorphic to $V_{\Z\gm - j\gm/k}$,
for the operator $\gm(0)$ acts on
$e^{n\gm - j\gm/k} \in V_{\Z\gm - j\gm/k}$ as
a scalar $\la \gm, n\gm - j\gm/k \ra = 2kn - 2j$ and
commutes with $\gm(m)$, $m \in \Z$.
Hence $V^j \cong T^j \otimes N^j$, where $T^j$ is a simple
$T^0$-module isomorphic to $V_{\Z\gm - j\gm/k}$.
In particular, $V^0 \cong T^0 \otimes N^0$ as vertex operator algebras.
Since $V^0$ is simple and $V^j$, $1 \le j \le k-1$ are simple $V^0$-modules,
the following lemma holds.

\begin{lem}\label{lem:V0}
$V^0 \cong T^0 \otimes N^0$ as vertex operator algebras and
$N ^0$ is a simple vertex operator algebra.
Moreover, $N^j$, $1 \le j \le k-1$ are simple $N^0$-modules.
\end{lem}

The weight $1$ subspace of $V^0$ is $V^0_{(1)} = \C\gm(-1)\1$ and
we have $\om_2 V^0_{(1)} = 0$. 
Hence, $V^0$ possesses a nonzero
invariant bilinear form by \cite[Theorem 3.1]{Li1} and so the following lemma holds.

\begin{lem}\label{lem:self-dual}
The vertex operator algebra $V^0$ is self-dual.
\end{lem}

The top level $N^j(0)$ of $N^j$ is of weight $j(k-j)/k$ by
our hypothesis. Hence the weight of the top level
$V^j(0) = T^j(0) \otimes N^j(0)$ of $V^j$ is $j$ if $0 \le j < k/2$ and
$k-j$ if $k/2 < j \le k-1$.
In the case $k$ is even and $j = k/2$, the weight of $V^j(0)$ is $k/2$.

Now, $V_{(n)} = 0$ for $n < 0$ and $V_{(0)} = V^0_{(0)} = \C\1$.
Moreover, $V_{(1)} = V^0_{(1)} + V^{k-1}(0) + V^1(0)$ and
$V^0_{(1)} = \C \gm(-1)\1$, for we are assuming that $k \ge 3$. 
Note that $V^{k-1}(0) = \C e^{\gm/k} \otimes N^{k-1}(0)$ and
$V^1(0) = \C e^{-\gm/k} \otimes N^1(0)$.
Also, $u_n v \in V^0$ for $u \in V^{k-1}$ and $v \in V^1$
by \eqref{eq:unv}.
By Lemmas \ref{lem:nonzero-action} and \ref{lem:self-dual},
we can choose $E \in V^{k-1}(0)$ and
$F \in V^1(0)$ such that $E_1 F = k\1$. Then
\eqref{eq:intertwiner-lattice1} implies that $E_0 F = \gm(-1)\1$.
Let $H = \gm(-1)\1$.

\begin{lem}\label{lem:OPE}
$(1)$ $H_0 H = 0$, $H_1 H = 2k\1$.

$(2)$ $H_0 E = 2E$, $H_1 E = 0$, $H_0 F = -2F$, $H_1 F = 0$.

$(3)$ $E_0 F = H$, $E_1 F = k\1$.

$(4)$ $E_0 E = E_1 E = F_0 F = F_1 F = 0$.
\end{lem}

\begin{proof}
Since $H_n = \gm(n) \otimes 1$ for $n \in \Z$ as an operator
on the $T^0 \otimes N^0$-module $T^j \otimes N^j$ and since
$\gm(n) e^{\pm \gm/k} = \pm 2 \delta_{n,0} e^{\pm \gm/k}$
if $n \ge 0$, (1) and (2) hold. We have chosen $E$ and $F$ so that
(3) holds. By \cite[(8.9.9)]{FLM}, $E_0 E = F_0 F = 0$.
We also have $E_1 E \in V^{k-2} \cap V_{(0)} = 0$ and
$F_1 F \in V^2 \cap V_{(0)} = 0$. Hence (4) holds.
\end{proof}

We want to show that the vertex operator algebra $V$ is generated by $H$, $E$ and $F$.
Let $U$ be the vertex operator subalgebra of $V$ generated by
$H$, $E$ and $F$.
Note that $A_n B = 0$ with $A, B \in \{ H, E, F\}$ and $n \ge 2$, for
the weight of $A_n B$ is $-n+1$.
Then Lemma \ref{lem:OPE} implies that the component operators
$H_n$, $E_n$, $F_n$, $n \in \Z$ give a level $k$ representation of
$\hsl$ under the correspondence
\begin{equation*}
h(n) \leftrightarrow H_n, \qquad
e(n) \leftrightarrow E_n, \qquad
f(n) \leftrightarrow F_n.
\end{equation*}

Since $A_n\1 = 0$ if $n \ge 0$ and since
$A_{-1}\1 = A$ for $A \in \{H,E,F\}$, the map
\begin{equation}\label{eq:map-to-U}
h(-1)\1 \mapsto H, \qquad
e(-1)\1 \mapsto E, \qquad
f(-1)\1 \mapsto F
\end{equation}
lifts to a surjective homomorphism $\varphi: V(k,0) \rightarrow U$
of vertex operator algebras by the universality of the Weyl
module $V(k,0)$. The image $\varphi(\J)$ of the maximal ideal
$\J$ of $V(k,0)$ is a maximal ideal of $U$ and the
quotient vertex operator algebra $U/\varphi(\J)$ is isomorphic
to $L(k,0) = V(k,0)/\J$. Hence there is a surjective homomorphism
$\psi: U \rightarrow L(k,0)$ of vertex operator algebras such that
$\psi(H) = h(-1)\1$, $\psi(E) = e(-1)\1$, and $\psi(F) = f(-1)\1$.
Recall that we use the same symbols to denote elements
of $V(k,0)$ and their images in $L(k,0)$.

Since 
$H_n = \gm(n) \otimes 1$, we have
\begin{equation*}
\gm(-n_1)\gm(-n_2)\cdots\gm(-n_r)\1 \in U
\end{equation*}
for $n_1 \ge n_2 \ge \cdots \ge n_r \ge 1$, $r = 0, 1, 2, \ldots$.

\begin{lem}\label{lem:N0-in-U}
$N^0 \subset U$.
\end{lem}

\begin{proof}
The image of the conformal vector $\om_\mraff$ of $V(k,0)$ under
$\varphi$ is
\begin{equation*}
\varphi(\om_\mraff) = \frac{1}{2(k+2)}
\Big( \frac{1}{2}H_{-1}H + E_{-1}F + F_{-1}E \Big),
\end{equation*}
which is contained in $V^0$ by \eqref{eq:unv}. We also have
\begin{equation*}
\varphi(\om_h) = \frac{1}{4k}H_{-1}H \in T^0.
\end{equation*}

Note that $N^0 = \{ v \in V^0 \,|\, H_m v = 0 \textrm{ for } m \ge 0 \}$.
Since $\om_\mrpara$ is the conformal vector of $N_0$ (cf. \eqref{eq:Nlambda})
of central charge $2(k-1)/(k+2)$,
its image $\varphi(\om_\mrpara)$ is a Virasoro element of
the same central charge. 
Moreover, $h(m) \om_\mrpara = 0$
implies $H_m \varphi(\om_\mrpara) = 0$, $m \ge 0$.
Thus $\varphi(\om_\mrpara)$ is contained in $N^0$.
Since $N^0_{(2)} = \C\om^2$ by our hypothesis, we conclude that
$\varphi(\om_\mrpara) = \om^2$. In particular, $\om^2 \in U$.

Next, we consider the image of $W^3$ \eqref{eq:W3} under the homomorphism
$\varphi$,
\begin{align*}
\varphi(W^3) &= k^2 H_{-3}\1 + 3kH_{-2}H_{-1}\1 + 2(H_{-1})^3\1
- 6kH_{-1}E_{-1}F_{-1}\1\\
& \quad + 3k^2E_{-2}F_{-1}\1 - 3k^2E_{-1}F_{-2}\1.
\end{align*}

As in the case of $\om_\mrpara$, we have
$\varphi(W^3) \in V^0$ by \eqref{eq:unv} and furthermore,
$\varphi(W^3) \in N^0$, for $W^3 \in N_0$.
Recall that $W^3$ is a Virasoro singular vector with
respect to the conformal vector $\om_\mrpara$ of $N_0$.
Hence $\varphi(W^3)$ is a Virasoro singular vector with respect to the conformal
vector $\om^2$ of $N^0$. Now, the weight $3$ subspace $N^0_{(3)}$
is of dimension $2$ by our hypothesis. Thus $\om^2_0 \om^2$ and
$\varphi(W^3)$ form a basis of $N^0_{(3)}$. Hence the lemma holds,
for we are assuming that the vertex operator algebra $N^0$ is
generated by $N^0_{(2)}$ and $N^0_{(3)}$.
\end{proof}

\begin{lem}\label{lem:epmgmN0-inU}
$e^{\pm \gm} \otimes N^0 \subset U$.
\end{lem}

\begin{proof}
Since $H_n = \gm(n) \otimes 1$ and since 
$V = \bigoplus_{j=0}^{k-1} V^j$
with $V^j \cong T^j \otimes N^j$, it follows from \eqref{def:Vj} that
\begin{equation*}
e^{\pm \gm} \otimes N^0  = \{ v \in V \,|\, H_n v = \pm 2k\dl_{n,0}v
\text{ for } n \ge 0\}.
\end{equation*}

Now, $e(-1)^k\1 \not\in \J$ and so
$\varphi(e(-1)^k\1) = (E_{-1})^k\1$ is a nonzero element of $U$.
Note also that
\begin{equation*}
h(n) e(-1)^k\1 = 2k\dl_{n,0}e(-1)^k\1, \quad n \ge 0
\end{equation*}
in the Weyl module $V(k,0)$. Taking the image under the
homomorphism $\varphi$, we have
\begin{equation*}
H_n (E_{-1})^k\1 = 2k\dl_{n,0}(E_{-1})^k\1, \quad n \ge 0.
\end{equation*}

This implies that $(E_{-1})^k\1$ is a nonzero element of
$e^\gm \otimes N^0$. Then we have $e^\gm \otimes N^0 \subset U$
by Lemma \ref{lem:N0-in-U}. Replacing $e(-1)$ with $f(-1)$ and
$E$ with $F$ in the above argument, we can also show that
$e^{-\gm} \otimes N^0 \subset U$.
\end{proof}

\begin{lem}\label{lem:U-equal-V}
$U = V$.
\end{lem}

\begin{proof}
The vertex operator algebra $V_{\Z\gm}$ is generated by $e^\gm$
and $e^{-\gm}$. Hence Lemma \ref{lem:epmgmN0-inU} implies that $U$ contains
$V^0$. Recall that $V^j$ is a simple $V^0$-module. Since
$F \in V^1$, it follows that $V^1 \subset U$. Then $V^j \subset U$
for all $j$ by \cite[Proposition 11.9]{DL}, and we have $U = V$
as desired.
\end{proof}

Since $V$ is a simple vertex operator algebra, Lemma \ref{lem:U-equal-V} implies
the following theorem.

\begin{thm}\label{th:Sec3}
$V \cong L_{\hsl}(k,0)$.
\end{thm}

\section{A characterization of $K(\mathfrak{sl}_2,k)$}\label{sec:characterization-M0}
In this section we apply the results of Section \ref{sec:characterization-Lk0} to
obtain a characterization of the parafermion vertex operator algebra $K(\mathfrak{sl}_2,k)$
associated with $\mathfrak{sl}_2$. 
Let $k \ge 3$ be an integer. 
Throughout this section we assume the following hypothesis.

\begin{hyp}\label{hypothesis:2}
$(1)$ $N^0$ is a simple, self-dual, rational and $C_2$-cofinite vertex operator algebra 
of CFT-type with central charge $2(k-1)/(k+2)$.

$(2)$ $\ch N^0 = 1 + q^2 + 2q^3 + \cdots$.

$(3)$ $N^0$ is generated by $N^0_{(2)}$ and $N^0_{(3)}$.

$(4)$ There exist simple current $N^0$-modules 
$N^j$, $1 \le j \le k-1$ 
such that the top weight of $N^j$ is $j(k-j)/k$ and 
the fusion rules among $N^j$'s are
\begin{equation*}
N^i \times N^j = N^{i+j}, \quad 0 \le i,j \le k-1.
\end{equation*}
Here the indices $i$, $j$ are considered to be modulo $k$. 

$(5)$ Any simple $N^0$-module except $N^0$ itself has positive top weight.

\end{hyp}

Let $V_{\Z\gm - j\gm/k}$ 
be as in Section \ref{subsec:latticeVOA}. Thus $\la \gm,\gm \ra = 2k$. 
Let
\begin{equation*}
V^j = V_{\Z\gm - j\gm/k} \ot N^j, \quad 0 \le j \le k-1
\end{equation*}
be a tensor product of vector spaces 
$V_{\Z\gm - j\gm/k}$ and $N^j$.
Then $V^0 = V_{\Z\gm} \ot N^0$ carries a structure of vertex operator algebra. 
In fact, $V^0$ is a simple, self-dual, rational and $C_2$-cofinite vertex operator algebra 
of CFT-type with central charge $3k/(k+2)$ by our assumption on $N^0$. 
Moreover, any simple $V^0$-module except $V^0$ itself has positive top weight.  
The $V^j$, $0 \le j \le k-1$ 
are simple $V^0$-modules 
\cite[Section 4.7]{FHL}. 
These simple modules are simple current $V^0$-modules and 
the fusion rules among them are 
\begin{equation}\label{eq:V0-fusion}
V^i \times V^j = V^{i+j}, \quad 0 \le i,j \le k-1
\end{equation}
by \eqref{eq:VZgm-fusion} and Hypothesis \ref{hypothesis:2}. 

The top weight of the $V^0$-module $V^j$ is a sum of
those of the $V_{\Z\gm}$-module $V_{\Z\gm - j\gm/k}$ and
the $N^0$-module $N^j$, so that it is $j$ if
$0 \le j < k/2$, $k-j$ if $k/2 < j \le k-1$ and $k/2$ if
$k$ is even and $j = k/2$. In particular, $V^j$ has integral weight.

Therefore, $\bop_{j=0}^{k-1} V^j$ has a vertex operator algebra structure 
by Theorem \ref{thm:SCE}. 
It is a $\Z_k$-graded simple current extension of $V^0$. 

Now, we can apply Theorem \ref{th:Sec3} to conclude that 
the vertex operator algebra $\bop_{j=0}^{k-1} V^j$ 
is isomorphic to $L_{\hsl}(k,0)$. 
Thus the following theorem holds. 

\begin{thm}\label{thm:Sec4}
The space
$\bop_{j=0}^{k-1} V^j$ is a $\Z_k$-graded simple current extension of $V^0$ 
and it is isomorphic to the vertex operator algebra $L_{\hsl}(k,0)$, where 
$V^j = V_{\Z\gm - j\gm/k} \ot N^j$ with $\la \gm,\gm \ra = 2k$ and 
$N^j$ being as in Hypothesis \ref{hypothesis:2}. 
\end{thm}

Since $N^0$ is the commutant of $V_{\Z\gm}$ in the vertex operator algebra 
$\bop_{j=0}^{k-1} V^j$ and since
the parafermion vertex operator algebra $K(\mathfrak{sl}_2,k)$ is the
commutant of $V_{\Z\gm}$ in $L_{\hsl}(k,0)$, 
the following theorem is a consequence of Theorem \ref{thm:Sec4}.

\begin{thm}\label{thm:characterize-paraf}
Let $N^0$ be as in Hypothesis \ref{hypothesis:2}. Then 
$N^0$ is isomorphic to the parafermion vertex operator
algebra $K(\mathfrak{sl}_2,k)$ associated with $\mathfrak{sl}_2$ for $k \ge 3$.
\end{thm}

\section{$W$-algebras}

In this section we recall some results on $W$-algebras.

\subsection{$W$-algebra 
$\W_\ell(\mathfrak{sl}_k)$}\label{subsec:W-algebra}
Let $\g$ be a finite-dimensional simple Lie algebra,
$(~|~)$ the normalized invariant inner product of $\g$,
and let $ \widehat{\g}$ be the affine Lie algebra
associated with $\g$ and $(~|~)$:
\begin{align*}
 \widehat{\g}=\g\otimes\mathbb{C}[t,t^{-1}]\oplus \mathbb{C} C,
\end{align*}
where
$C$ is the central element.
Let $\widehat{\h}=\h\oplus \mathbb{C} C$ be the 
Cartan subalgebra of $\widehat{\g}$,
where $\finh$ is a Cartan subalgebra of $\g$.
Let $\wh^*=\h^*\oplus \C \Lam_0$ be the dual of $\h$,
where $\Lam_0(C)=1$, $\Lam_0(\h)=0$.

Let $V_{\wg}(\ell,0)
=U(\widehat{\g})\otimes_{U(\g[t]\oplus \C C)}\C_{\ell}$ be 
the {\em universal affine vertex algebra} associated with $\g$ at level $\ell \in \C$.
Here $\g \ot \C[t]$
acts as $0$  and $C$ acts as $\ell$ on $\C_{\ell}$ .

For a weight $\lam$ of $\g$
denote by $L_{\widehat{\g}}(\ell,\lam)$ the
simple highest weight module for 
$\widehat{\g}$ with highest weight
$\widehat{\lam}_{\ell} = \lam+\ell\Lam_0$.
The vacuum simple module
$L_{\widehat{\g}}(\ell,0)$ is  a quotient vertex algebra
of $V_{\wg}(\ell,0)$
and called the simple {\em affine vertex algebra
associated with $\fing$ at level $\ell$}.

Let $\W^{\ell}(\g)$
be the 
$W$-algebra associated with
$\fing$ and its principal nilpotent element
at a non-critical level $\ell$ 
defined by the quantized Drinfeld-Sokolov
reduction 
\cite{FF90}:
\begin{align*}
\W^{\ell}(\g)=H_{DS}^0(V_{\wg}(\ell,0)),
\end{align*}
where $H^{\bullet}_{DS}(M)$ is the  cohomology of the BRST complex
for the quantized Drinfeld-Sokolov reduction with coefficient
in a $\widehat{\g}$-module $M$ (\cite{FF90}).
Denote by 
$\W_{\ell}(\g)$ the unique simple quotient of $\W^{\ell}(\g)$.

Later we shall set $\fing=\mf{sl}_k$, in which case
 $\W^{\ell}(\fing)$ is isomorphic to the {\em $W_k$-algebra}
defined by Fateev and Lukyanov \cite{FL88} (cf. \cite{AM}).

The level $\ell$ is called  {\em admissible}
if
$L_{\widehat{\g}}(\ell,0)$
 is an admissible representation \cite{KacWak89}
of
$\widehat{\g}$.
If this is the case,  
the associated variety 
$X_{L_{\widehat{\g}}(\ell,0)}$
of $L_{\widehat{\g}}(\ell,0)$ is contained 
in the nullcone
$\mathcal{N}$ of $\g$ (\cite{Ara12}),
where 
$X_{L_{\widehat{\g}}(\ell,0)}$ is the maximal spectrum of Zhu's
$C_2$-algebra of $L_{\widehat{\g}}(\ell,0)$ (\cite{Ara12}). 
An admissible number $\ell$
is called {\em non-degenerate}
if $X_{L_{\widehat{\g}}(\ell,0)}=\mathcal{N}$ (\cite{FKW,Ara12}).
For an admissible level $\ell$,
we have 
(\cite{Ara10})
\begin{align*}
H^{\bullet}_{DS}
(L_{\widehat{\g}}(\ell,0))\ne 0 \quad \iff \quad \ell\text{ is non-degenerate}.
\end{align*}

The following assertion was conjectured by Frenkel, Kac and Wakimoto \cite{FKW}.
\begin{thm}[\cite{Arakawa,Ara10,A2012Dec}]\label{Th:rationality-C2-for-BRST-W}
Let $\ell$ be a non-degenerate admissible level.
Then we have the isomorphism
\begin{align*}
\W_{\ell}(\g)\cong H^0_{DS}(L_{\widehat{\g}}(\ell,0))
\end{align*}
of vertex algebras.
Moreover, $\W_{\ell}(\g)$ is self-dual, rational and $C_2$-cofinite.
\end{thm}

The rational and $C_2$-cofinite $W$-algebras
appearing in Theorem  \ref{Th:rationality-C2-for-BRST-W}
are called {\em minimal series $W$-algebras}.
In the case that 
$\g=\mf{sl}_2$,
 they are exactly the Virasoro vertex operator algebras
which belong to minimal series representations \cite{BPZ84}
 of the Virasoro algebra.
 
The conjectural classification  \cite{FKW} of simple modules 
for minimal series $W$-algebras
was also  established  in \cite{A2012Dec}.
In \cite{A-vEk} 
it was shown that
this together with 
Theorem \ref{Th:rationality-C2-for-BRST-W}
verifies the conjectual fusion rules
of minimal series $W$-algebras
obtained in \cite{FKW}.
In the next subsection we shall describe these
results more precisely 
in the cases that we are interested in this article.

\subsection{The case $\g=\mf{sl}_k$}

Now, we set $\g=\mf{sl}_k$. 
One knows \cite{FreBen04} that
$\W^{\ell}(\g)$ is freely generated by homogeneous
elements of weight $d_1+1,\dots, d_{\on{rank}\g}+1$,
where $d_1,\dots, d_{\on{rank}\g}$ 
is the exponents of $\g$.
In particular $\W^{\ell}(\mf{sl}_k)$
is freely generated by homogeneous elements of weight
$2,3,\dots,k$.
Hence
\begin{align*}
\on{ch}\W^{\ell}(\mathfrak{sl}_k)=\prod_{i=1}^{k-1}\prod_{j=1}^{\infty}(1-q^{i+j+1})^{-1}
=1+q^2+2q^3+\dots
\end{align*}

\begin{prop}[\cite{FL88}]\label{Pro:generator-of-W}
For any non-cricial level $\ell$,
$\W^{\ell}(\mf{sl}_k)$ is generated by its weight $2$ and weight $3$
subspaces as a vertex algebra.
 \end{prop}
We include    a proof of this fact
in Appendix \ref{appen}.

Note that  we have
\begin{align*}
\W^{\ell}(\mf{sl}_k)\cong \W^{\ell'}(\mf{sl}_k)\quad \text{if }(\ell+k)(\ell'+k)=1
\end{align*}
 (\cite{FL88}, see
also \cite{AM}), which is the special case of the 
Feigin-Frenkel duality \cite{FeiFre92}.

The level
$\ell$ is a non-degenerate admissible number 
for $\widehat{\g}=\widehat{\mf{sl}}_k$ if and only if
\begin{align*}
\ell+k=\frac{p}{q}\quad \text{with }
p,q\in \mathbb{N},\ (p,q)=1,\
p,q\geq k.
\end{align*}
Set
\begin{align*}
\W_{p,q}(\mf{sl}_k)=\W_{p/q-k}(\mf{sl}_k)=\W_{q/p-k}(\mf{sl}_k)
\end{align*}
for 
$p/q-k$ with
 $p,q\in \mathbb{N}$, $(p,q)=1$,
$p,q\geq k$.
The central charge of  $\W_{p,q}(\mf{sl}_k)$ is given by
\begin{align*}
c_{p,q}=-\frac{(k-1)((k+1)p-kq)(kp-(k+1)q)}{pq}.
\end{align*}
Note that
\begin{align}
c_{k+1,k+2} = \frac{2(k-1)}{k+2}.
\end{align}
The simple modules
of $\W_{p,q}(\mf{sl}_k)$
are parametrized by the set
\begin{align*}
I_{p,q}=(\hat P_+^{p-k}\times \hat P_+^{q-k})/\widetilde{W}_+,
\end{align*}
where
$\widehat{P}_+^{m}$ denotes the set of integral dominant weights 
of $\widehat{\g}$ of level $m$
and $\widetilde{W}_+$ is the subgroup of the extended affine Weyl group 
consisting of elements of length zero which acts diagonally 
on the set $\hat P_+^{p-k}\times \hat P_+^{q-k}$.
By \cite[Theorem 10.4]{A2012Dec} and \cite[Remark 9.1.8]{Arakawa},
$$\{H_{DS}^0(L_{\wg}(\ell,\lam-\frac{p}{q}\mu))\mid 
[(\hat{\lam}_{p-k},\hat{\mu}_{q-k})]\in I_{p,q},\ \lam,\mu\in \h^*\}$$
gives a complete set of representatives 
of the isomorphism classes of simple $\W_{p,q}(\mf{sl}_k)$-modules.

\subsection{Simple modules and fusion rules of $\W_{k+1,k+2}(\mf{sl}_k)$}
\label{subsec:W-pq}
Consider the special case that
$q=k+1$.
Then
we have a bijection
\begin{align*}
\hat P_+^{p-k}\isomap I_{p,k+1},\quad  \hat{\lam}_{p-k}\mapsto [(\hat{\lam}_{p-k},\hat{0}_{1})].
\end{align*}
Therefore, by putting
\begin{align*}
\M(\Lam)=H^0_{BRST}(L_{\widehat{\mathfrak{sl}}_k}(\ell,\bar \Lam)),
\end{align*}
where  $\bar \Lam$ is the restriction of $\Lam$ to $\mf{h}$,
the set 
$\{\M(\Lam)\mid
\Lam\in \hat{P}^{p-k}_+\}$ 
gives a complete set of representatives 
of the isomorphism classes of simple $\W_{p,k+1}(\mathfrak{sl}_k)$-modules.

Let
$\mc{A}^m_{\fing}=\{[L_{\widehat{\mf{sl}}_k(m,\lam)}]\mid
\lam\in {P}^m_+
\}$ be the fusion algebra 
for $\widehat{\mf{sl}}_k$ at level $m$,
and let
$\mc{A}^{p,q}_{\W}$ be the fusion algebra
of
$\W_{p,q}(\mf{sl}_k)$.
Note that $\W_{p,q}(\mf{sl}_k) = \W_{q,p}(\mf{sl}_k)$ and so 
$\mc{A}^{p,q}_\W\cong \mc{A}^{q,p}_{\W}$.

The following assertion is the special case of the 
fusion rule of $\W_{p,q}(\mf{sl}_k)$ computed in \cite{FKW}.
\begin{thm}
Let $p$ be an integer such that
$p\geq k$, $(p,k+1)=1$.
Then
the assignment
\begin{align*}
[L_{\widehat{\mf{sl}}_k}(\ell,  \lam)] \mapsto [\M(\hat{\lam}_{\ell})]
\end{align*}
gives the isomorphism of 
fusion algebras
$\mc{A}^{p-k}_{\fing}\isomap \mc{A}^{p,k+1}_{\W}$.
\end{thm}

Now we set $p=k+2$.
Then
$$\{\M(\Lam_i+\Lam_j )\mid 0\leq i \leq j\leq k-1\}$$
gives a complete set of representatives 
of the isomorphism classes of simple $\W_{k+2,k+1}(\mf{sl}_k)$-modules.
The top level of 
$\M(\Lam_i+\Lam_j)$ is 
one dimensional 
with weight given by 
\begin{align}
 \frac{-i^2+i (k (2 k+3)-2 j (k+1))+j (k-j)}{2 k (k+2)}.
 \label{eq:top-wt-Lambdaj-Lambdai}
\end{align}

\begin{cor}
$\M(\Lam_i+\Lam_j)$ is a simple current
module for $\W_{k+1,k+2}(\mf{sl}_k)$
if and only if $i=j$.
We have
\begin{equation}\label{eq:W-alg-fusion}
\M(2\Lambda_p) \times \M(\Lambda_j + \Lambda_i) = \M(\Lambda_{j+p},\Lambda_{i+p})
\end{equation}
for $0 \le p \le k-1$ and $0 \le i,j \le k-1$,
where the index 
is considered  to be modulo $k$.
 In particular, 
\begin{equation}\label{eq:W-alg-fusion2}
\M(2\Lambda_p) \times \M(2\Lambda_q) = \M(2\Lambda_{p+q}).
\end{equation}
\end{cor}
We remark that
$$\M(\Lam_i+\Lam_j)'\cong \M(\Lam_{-i}+\Lam_{-j}),$$
where $\M(\Lam_i + \Lam_j)'$ is the dual module of $\M(\Lam_i + \Lam_j)$ 
(see \cite[Theorem 5.5.4]{Arakawa}).

\medskip
Here we summarize some of the properties of 
$\M(2\Lambda_0) = \W_{k+1,k+2}(\mf{sl}_k)$.

(1) $\M(2\Lambda_0)$ is a simple, self-dual, rational and $C_2$-cofinite vertex operator algebra 
of CFT-type with central charge $2(k-1)/(k+2)$.

(2) $\ch \M(2\Lambda_0) = 1 + q^2 + 2q^3 + \cdots$.

(3) $\M(2\Lambda_0)$ is generated by $\M(2\Lambda_0)_{(2)}$ and $\M(2\Lambda_0)_{(3)}$.

(4) $\M(2\Lambda_j)$, $0 \le j \le k-1$ are simple current 
$\M(2\Lambda_0)$-modules and the fusion rules among them are 
$\M(2\Lambda_i) \times \M(2\Lambda_j) = \M(2\Lambda_{i+j})$.

(5) The top weight of $\M(2\Lambda_j)$ is $j(k-j)/k$.

(6) Any simple $\M(2\Lambda_0)$-module except $\M(2\Lambda_0)$ itself has positive top weight.

\begin{rmk}
The pair $(i,j)$ of the indices $i$ and $j$ for a complete set of representatives 
of the isomorphism classes of simple $M^0$-modules $M^{i,j}$ runs over the range 
$0 \le j < i \le k$, while that for $\M(2 \Lambda_0)$-modules 
$\M(\Lambda_j + \Lambda_i)$ runs over the range $0 \le j \le i \le k-1$. 
Let $i' = k-i+j$. Then $0 \le j \le i' \le k-1$ if and only if $0 \le j < i \le k$. 
Thus these two sets of parameters $(i,j)$'s are related as
\begin{equation*}
\{ (i,j) | 0 \le j \le i \le k-1 \} = \{ (k-i+j,j) | 0 \le j < i \le k \}.
\end{equation*}
\end{rmk}

The following lemma will be used in Section 6.

\begin{lem}\label{lem:top-weight}
The top weight of the simple 
$K(\mathfrak{sl}_2,k)$-module $M^{i,j}$ 
is equal to that of  the simple 
$\M(2\Lambda_0)$-module 
$\M(\Lambda_j + \Lambda_{j-i})$ for $0 \le i \le k$, $0 \le j \le k-1$.
\end{lem}

\begin{proof}
By \eqref{eq:top-weight-Mij} and \eqref{eq:top-wt-Lambdaj-Lambdai}
we see that the top weight of the simple $K(\mathfrak{sl}_2,k)$-module $M^{i,j}$ 
is equal to that of the simple 
$\M(2\Lambda_0)$-module 
$\M(\Lambda_j + \Lambda_{j-i})$ for $0 \le j < i \le k$. 
For a pair $(i,j)$ with $0 \le i \le j \le k-1$, we have $0 \le j-i < k-i \le k$. 
Hence the top weight of $M^{k-i,j-i}$ coincides with that of 
\begin{equation*}
\M(\Lambda_{j-i}+ \Lambda_{(j-i)-(k-i)}) = \M(\Lambda_{j-i} + \Lambda_{j-k}).
\end{equation*}

Since $M^{i,j}$ is isomorphic to $M^{k-i,j-i}$ as $M^0$-modules by \eqref{eq:equiv-Mij} and 
$\M(\Lambda_{j-i} + \Lambda_j) = \M(\Lambda_{j-k} + \Lambda_{j-i})$ 
the assertion holds for such a pair $(i,j)$ also.
\end{proof}

%

\section{Identification of $K(\mathfrak{sl}_2,k)$ and $\W_{k+1,k+2}(\mathfrak{sl}_k)$}
In this section we use the results of Section 4 to show that $K(\mathfrak{sl}_2,k)$ is isomorphic to 
the $(k+1,k+2)$-minimal series $W$-algebra $\W_{k+1,k+2}(\mathfrak{sl}_k)$.
We also discuss a correspondence of the simple modules for $K(\mathfrak{sl}_2,k)$ 
with those for $\W_{k+1,k+2}(\mathfrak{sl}_k)$. 
If $k=2$, it is well-known that both $K(\mathfrak{sl}_2,2)$ and $\W_{3,4}(\mathfrak{sl}_2)$ 
are isomorphic to the simple Virasoro vertex operator algebra with central charge $1/2$. 
So let us assume that $k\geq 3$.

By the properties of $\W_{k+1,k+2}(\mathfrak{sl}_k)$ described 
in Section \ref{subsec:W-pq}, we see that 
$\W_{k+1,k+2}(\mathfrak{sl}_k)$ satisfies the five conditions of 
Hypothesis \ref{hypothesis:2} for $N^0$ together with the simple current 
modules $\M(2\Lambda_j)$ for $N^j$, $0 \le j \le k-1$. 
Therefore, Theorem \ref{thm:characterize-paraf} implies 
the next theorem.

\begin{thm}\label{thm:identification}
The $(k+1,k+2)$-minimal series $W$-algebra $\W_{k+1,k+2}(\mathfrak{sl}_k)$ is isomorphic to 
the parafermion vertex operator algebra $K(\mathfrak{sl}_2,k)$ of type $\mathfrak{sl}_2$.
\end{thm}

Furthermore, it follows from Theorem \ref{thm:Sec4} that
$\bigoplus_{j=0}^{k-1} V_{\Z\gm - j\gm/k} \ot \M(2\Lambda_j)$ is a $\Z_k$-graded 
simple current extension of $V_{\Z\gm} \otimes  \M(2\Lambda_0)$ 
and it is isomorphic to $L_{\hsl}(k,0)$. 
That is,
\begin{equation}\label{eq:Lk0-dec-W}
L_{\hsl}(k,0) = \bigoplus_{j=0}^{k-1} V_{\Z\gm - j\gm/k} \ot \M(2\Lambda_j)
\end{equation}
as $V_{\Z\gm} \otimes M(2\Lambda_0)$-modules. 

\begin{cor}
The parafermion vertex operator algebra 
$K(\mf{sl}_2,k)$ is rational.
\end{cor}

It is known that
$K(\mf{sl}_2,k)$ is isomorphic to  
$\on{Com}_{L_{\widehat{\mf{sl}}_k}(1,0)\otimes L_{\widehat{\mf{sl}}_k(1,0)}}
(L_{\widehat{\mf{sl}}_k}(2,0))$
(\cite{Lam2}).
Therefore 
Theorem \ref{thm:identification}
immediately gives the following
assertion that has been conjectured in \cite{KacWak90,BS}.
\begin{cor}
We have the isomorphism
\begin{align*}
\W_{k+1,k+2}(\mf{sl}_k)\cong 
\on{Com}_{L_{\widehat{\mf{sl}}_k}(1,0)\otimes L_{\widehat{\mf{sl}}_k(1,0)}}
(L_{\widehat{\mf{sl}}_k}(2,0)).
\end{align*}
\end{cor}


For simplicity of notation we identify $M^0$ with $\M(2\Lambda_0)$, 
so that $M^0 = \M(2\Lambda_0)$. 
Then it follows from \eqref{eq:Lk0-dec}, \eqref{def:Mj} and \eqref{eq:Lk0-dec-W} that 
$M^j = \M(2\Lambda_j)$ for $0 \le j \le k-1$. 
This in particular implies that $M^j$, $0 \le j \le k-1$ are the simple current 
modules for $K(\mathfrak{sl}_2,k)$.  
The fusion rules among them are
\begin{equation}\label{eq:fusion-Mi-Mj}
M^i \times M^j = M^{i+j}
\end{equation}
by \eqref{eq:W-alg-fusion2}, which is compatible with \eqref{eq:VZgm-fusion}.


We have another description of $L_{\hsl}(k,0)$ as a $\Z_k$-graded 
simple current extension of $V_{\Z\gm} \otimes \M(2\Lambda_0)$. 
Indeed, let $N^j$ be as in Section 4. Then 
\begin{equation*}
N^{k-i} \times N^{k-j} = N^{k - (i+j)},
\end{equation*}
and so $N^{k-j}$, $0 \le j \le k-1$ satisfy the conditions of 
Hypothesis \ref{hypothesis:2} for $N^j$, $0 \le j \le k-1$. 
Therefore, we can apply the argument in Section 4 to $N^{k-j}$ in place of $N^j$ 
to conclude that $\bigoplus_{j=0}^{k-1} V_{\Z\gm - j\gm/k} \ot N^{k-j}$ 
is a simple current extension of $V^0$ isomorphic to $L_{\hsl}(k,0)$. 
Thus
\begin{equation}\label{eq:equivalence}
\begin{split}
\bigoplus_{j=0}^{k-1} V_{\Z\gm - j\gm/k} \ot N^j 
&\cong \bigoplus_{j=0}^{k-1} V_{\Z\gm - j\gm/k} \ot N^{k-j}\\
&= \bigoplus_{j=0}^{k-1} V_{\Z\gm + j\gm/k} \ot N^j
\end{split}
\end{equation}
as vertex operator algebras and we have
\begin{equation}\label{eq:Lk0-dec-W-2}
L_{\hsl}(k,0) = \bigoplus_{j=0}^{k-1} V_{\Z\gm - j\gm/k} \ot \M(2\Lambda_{k-j})
\end{equation}
as $V_{\Z\gm} \otimes \M(2\Lambda_0)$-modules. 
In \eqref{eq:Lk0-dec-W-2} we have $M^j = \M(2\Lambda_{k-j})$ for $0 \le j \le k-1$.

Recall the automorphism $\theta$ of the vertex operator algebra $L_{\hsl}(k,0)$ 
of order $2$ discussed in Section \ref{subsec:parafermion}. 
The automorphism $\theta$ transforms $V_{\Z\gm - j\gm/k}$ to $V_{\Z\gm + j\gm/k}$ 
and induces an automorphism of $V_{\Z\gm}$. 
Hence $\theta \otimes 1$ gives the isomorphism \eqref{eq:equivalence}. 
Note also that $1 \otimes \theta$ is an automorphism of $V_{\Z\gamma} \otimes M^0$ 
and that $\theta$ transforms $M^j$ to $M^j \circ \theta = M^{k-j}$ 
by \eqref{eq:theta-action}. 
Thus the isomorphism \eqref{eq:equivalence} is afforded by $1 \otimes \theta$ also
(cf. \cite[Lemmas 3.14 and 3.15]{SY}). 

The next proposition implies that there are only two ways of describing $L_{\hsl}(k,0)$ 
as a  $\Z_k$-graded simple current extension of $V_{\Z\gm} \otimes  \M(2\Lambda_0)$, 
namely \eqref{eq:Lk0-dec-W} and \eqref{eq:Lk0-dec-W-2}. 
\begin{prop}\label{prop:two-cases}
One of the following two cases occurs.

$(1)$ $M^j = \M(2\Lambda_j)$ for all $0 \le j \le k-1$.

$(2)$ $M^j = \M(2\Lambda_{k-j})$ for all $0 \le j \le k-1$.
\end{prop}

\begin{proof}
Among the simple current 
$K(\mathfrak{sl}_2,k)$-modules $M^j$, $1 \le j \le k-1$ 
(resp. $\W_{k+1,k+2}(\mf{sl}_k)$-modules 
$\M(2\Lambda_j)$, $1 \le j \le k-1$), 
only $M^1$ and $M^{k-1}$ (resp. $\M(2\Lambda_1)$ and $\M(2\Lambda_{k-1})$)  
have top weight $(k-1)/k$. 
Hence $M^1 = \M(2\Lambda_1)$ or  $M^1 = \M(2\Lambda_{k-1})$. 
By the fusion rules $M^p \times M^1 = M^{p+1}$ and 
$\M(2\Lambda_p) \times \M(2\Lambda_1) = \M(2\Lambda_{p+1})$, 
(1) holds if  $M^1 = \M(2\Lambda_1)$, and 
(2) holds if  $M^1 = \M(2\Lambda_{k-1})$.
\end{proof}

Next, we study the correspondence of the remaining simple modules for 
$K(\mathfrak{sl}_2,k)$ with those for 
$\W_{k+1,k+2}(\mf{sl}_k)$. 
For this purpose we first inspect $M^{i,j}$, $0 \le i \le k$, $0 \le j \le k-1$. 
Recall that the second index $j$ of $M^{i,j}$ is considered to be modulo $k$ 
(cf. Remark \ref{rmk:j-mod-k}). 
By Lemma \ref{lem:top-weight}, the top weight of 
$M^{i,j}$ is equal to that of $\M(\Lambda_j + \Lambda_{j-i})$ 
for $0 \le i \le k$, $0 \le j \le k-1$.

Since $M^j$ is a simple current $M^0$-module, we see from 
\eqref{eq:Lk0-dec}, \eqref{eq:Lki-dec} and the fusion rules \eqref{eq:VZgm-fusion} 
for simple $V_{\Z\gm}$-modules that the fusion rule of $M^p$ and  $M^{i,j}$ is 
\begin{equation}\label{eq:fusion-Mp-Mij}
M^p \times M^{i,j} = M^{i,j+p}, \quad 0 \le p, j \le k-1, \ 0 \le i \le k.
\end{equation}

Let 
\begin{equation}\label{def:Pij}
P(i,j) = k(i-2j) - (i-2j)^2 + 2k(i-j+1)j
\end{equation}
for $0 \le j \le i \le k$. Then the top weight of $M^{i,j}$ is $P(i,j)/2k(k+2)$. 
Since $M^{i,j} \cong M^{k-i,j-i}$ as $M^0$-modules by \eqref{eq:equiv-Mij}, 
the top weight of $M^{i,j}$ for $0 \le i \le j \le k-1$ 
is given by $P(k-i, j-i)/2k(k+2)$. 
We have $P(i,0) = P(i,i) = i(k-i)$ and
\begin{equation*}
P(i,j) - i(k-i) = 2(k+2)j(i-j) \ge 0
\end{equation*}
for $0 \le j \le i \le k$. Moreover,
\begin{equation*}
P(k-i, j-i) - i(k-i) = 2(k+2)(k-j)(j-i) > 0
\end{equation*}
for $0 \le i < j \le k-1$. Thus the following Lemma holds. 

\begin{lem}\label{lem:minimum-top-wt}
Let $1 \le i \le k$.  
Then the top weight of $M^{i,j}$ for $0 \le j \le k-1$ is 
at least $i(k-i)/2k(k+2)$ and it is equal to $i(k-i)/2k(k+2)$ if and only if 
$j = 0, i$. 
\end{lem}

We also note that $i(k-i)$ is monotone increasing with respect to $i$ for 
$0 \le i \le k/2$.

We shall show that $M^{i,j} = \M(\Lambda_j + \Lambda_{j-i})$ for 
all $0 \le j < i \le k$ in the case (1) of Proposition \ref{prop:two-cases}. 
Thus, assume that $M^j = \M(2\Lambda_j)$ for all $0 \le j \le k-1$. 
We consider a decomposition of $L(k,i)$ into a direct sum of simple 
$V_{\Z\gm} \otimes \M(2\Lambda_0)$-modules and compare it with 
\eqref{eq:Lki-dec}.

The top weight of $V_{\Z\gm + (i-2j)\gm/2k} \otimes M^{i,j}$ is a sum of 
the top weight of $V_{\Z\gm + (i-2j)\gm/2k}$ and that of $M^{i,j}$. 
Hence we have 
\begin{equation}\label{eq:top-wt-difference}
(\text{top weight of } V_{\Z\gm + (i-2)\gm/2k} \otimes M^{i,i+1} ) 
- (\text{top weight of } V_{\Z\gm + i\gm/2k} \otimes M^{i,0} )  = \frac{k-2i}{k}
\end{equation}
for $0 \le i \le k$. 
In the range of $1 \le i \le k/2$, the difference $(k-2i)/k$ can be
an integer only if $k$ is even and $i = k/2$. 

Let 
\begin{equation}\label{def:S}
\begin{split}
\CS 
&= \{ M^{i,j} | 0 \le j < i \le k \}\\
&= \{ \M(\Lambda_j + \Lambda_i) | 0 \le j \le i \le k-1 \}
\end{split}
\end{equation}
be a complete set of representatives of the isomorphism classes of simple modules 
for $M^0 = \M(2\Lambda_0)$. 
Using the isomorphism $M^{i,j} \cong M^{k-i,j-i}$, we can arrange the representatives 
$M^{i,j}$'s so that 
\begin{equation}\label{eq:S-2-odd}
\CS = \{ M^{i,j} |  0 \le i \le (k-1)/2, 0 \le j \le k-1 \}
\end{equation}
if $k$ is odd, and
\begin{equation}\label{eq:S-2-even}
\CS = \{ M^{i,j} |  0 \le i \le k/2 - 1, 0 \le j \le k-1 \} 
\cup \{ M^{k/2, j} | 0 \le j \le k/2 - 1 \}
\end{equation}
if $k$ is even. In the case $k$ is even, we note that 
$M^{k/2, j} \cong M^{k/2, k/2 + j}$ for $ 0 \le j \le k/2 - 1$.

Set 
\begin{equation*}
\CT_i = \{ M^{i,j} |  0 \le j \le k-1 \}
\end{equation*}
for $0 \le i \le \lfloor (k-1)/2 \rfloor$, 
where $\lfloor (k-1)/2 \rfloor$ denotes the largest 
integer which does not exceed $(k-1)/2$. 

Moreover, set
\begin{equation*}
\CT_{k/2} = \{ M^{k/2,j} |  0 \le j \le k/2 - 1 \}
\end{equation*}
if $k$ is even. Set 
\begin{equation*}
\CS_p = \bigcup_{i=p}^{\lfloor k/2 \rfloor} \CT_i
\end{equation*}
for $0 \le p \le \lfloor k/2 \rfloor$. Then $\CS_0 = \CS$. 

We shall  establish an identification 
of $M^{i,j} \in \CT_i$ and $\M(\Lambda_j + \Lambda_{j-i}) \in \CT_i$, 
$0 \le j \le k-1$ ($0 \le j \le k/2 - 1$ if $k$ is even and $i = k/2$) 
for  $i = 1, 2, \ldots, \lfloor k/2 \rfloor$ inductively.
Note that $M^j = \M(2\Lambda_j)$, $0 \le j \le k-1$ by our assumption, 
that is, the identification is given for $i = 0$ and \eqref{eq:Lk0-dec-W} holds. 

First, we discuss the case $i = 1$. 
By Lemma \ref{lem:minimum-top-wt}, we see that the top weight of a simple 
module $M \in \CS_1$ is at least $(k-1)/2k(k+2)$ and it is equal to $(k-1)/2k(k+2)$ 
if and only if $M = M^{1,0}$ or $M^{1,1}$. Then Lemma \ref{lem:top-weight} implies that 
one of the following two cases occurs: 
$M^{1,0} = \M(\Lambda_0 + \Lambda_{k-1})$ and $M^{1,1} = \M(\Lambda_0 + \Lambda_1)$, or 
$M^{1,0} = \M(\Lambda_0 + \Lambda_1)$ and $M^{1,1} = \M(\Lambda_0 + \Lambda_{k-1})$.

Suppose $M^{1,0} = \M(\Lambda_0 + \Lambda_1)$. Then since the fusion rule
\begin{equation}\label{eq:M10-times-M11}
\big( V_{\Z\gm - \gm/k} \ot \M(2\Lambda_1) \big) \times 
\big( V_{\Z\gm + \gm/2k} \ot \M(\Lambda_0 + \Lambda_1) \big)
= V_{\Z\gm - \gm/2k} \ot \M(\Lambda_1 + \Lambda_2) 
\end{equation}
holds by \eqref{eq:VZgm-fusion} and \eqref{eq:W-alg-fusion}, both 
$V_{\Z\gm + \gm/2k} \ot \M(\Lambda_0 + \Lambda_1)$ and 
$V_{\Z\gm - \gm/2k} \ot \M(\Lambda_1 + \Lambda_2) $ appear 
as direct summands in a decomposition \eqref{eq:Lki-dec} 
of $L(k,1)$ into a direct sum of 
simple $V_{\Z\gm} \otimes \M(2\Lambda_0)$-modules. 
However, the top weigh of $\M(\Lambda_1 + \Lambda_2)$ coincides with 
that of $M^{1,2}$ by Lemma \ref{lem:top-weight}, and so the difference of the 
top weight of these two direct summands is $(k-2)/k$ by \eqref{eq:top-wt-difference}, 
which is not an integer. 
This is a contradiction, for $L(k,1)$ is a simple $L(k,0)$-module. 
Therefore, $M^{1,0} = \M(\Lambda_0 + \Lambda_{k-1})$. 
Then by the fusion rules \eqref{eq:W-alg-fusion} and \eqref{eq:fusion-Mp-Mij}, 
we obtain $M^{1,j} = \M(\Lambda_j + \Lambda_{j-1})$ for $0 \le j \le k-1$. 
Thus the identification of simple modules contained in $\CT_1$ holds.

Next, let $p$ be an integer such that $2 \le p \le \lfloor (k-1)/2 \rfloor$ and 
assume that the identification of simple modules contained in 
$\CT_i$, $0 \le i \le p-1$ holds. 

We replace $1$ with $p$ in the above argument. 
By Lemma \ref{lem:minimum-top-wt}, the top weight of a simple 
module $M \in \CS_p$ is at least $p(k-p)/2k(k+2)$ and it is equal to $p(k-p)/2k(k+2)$ 
if and only if $M = M^{p,0}$ or $M^{p,p}$. 
Hence one of the following two cases occurs: 
$M^{p,0} = \M(\Lambda_0 + \Lambda_{k-p})$ and $M^{p,p} = \M(\Lambda_0 + \Lambda_p)$, 
or 
$M^{p,0} = \M(\Lambda_0 + \Lambda_p)$ and $M^{p,p} = \M(\Lambda_0 + \Lambda_{k-p})$. 
Suppose $M^{p,0} = \M(\Lambda_0 + \Lambda_p)$. 
Then since 
\begin{equation}
\big( V_{\Z\gm - \gm/k} \ot \M(2\Lambda_1) \big) \times 
\big( V_{\Z\gm + p\gm/2k} \ot \M(\Lambda_0 + \Lambda_p) \big)
= V_{\Z\gm + (p-2)\gm/2k} \ot \M(\Lambda_1 + \Lambda_{p+1}) 
\end{equation}
by \eqref{eq:VZgm-fusion} and \eqref{eq:W-alg-fusion}, 
both $V_{\Z\gm + p\gm/2k} \ot \M(\Lambda_0 + \Lambda_p)$ and 
$V_{\Z\gm + (p-2)\gm/2k} \ot \M(\Lambda_1 + \Lambda_{p+1}) $ appear 
as direct summands in a decomposition of $L(k,p)$ into a direct sum of 
simple $V_{\Z\gm} \otimes \M(2\Lambda_0)$-modules. 
However, the top weigh of $\M(\Lambda_1 + \Lambda_{p+1})$ coincides with 
that of $M^{p,p+1}$, and so the difference of the 
top weight of these two direct summands is $(k-2p)/k$ by \eqref{eq:top-wt-difference}, 
which is not an integer. 
Thus $M^{p,0} = \M(\Lambda_0 + \Lambda_{k-p})$. 
Hence we have the identification of simple modules contained in  
$\CT_p$ by the fusion rules.

In the case $k$ is even and $p = k/2$, we have $\CS_{k/2} = \CT_{k/2}$. 
The minimum of the top weight of the simple modules contained in $\CT_{k/2}$ 
is $k/8(k+2)$ and it is attained only by $M^{2/k,0}$. 
Hence $M^{2/k,0} = \M(\Lambda_0 + \Lambda_{k/2})$ 
and the identification of simple modules contained in  
$\CT_{k/2}$ holds by the fusion rules. 
This completes the induction on $i$. 
Therefore, we conclude that $M^{i,j} = \M(\Lambda_j + \Lambda_{j-i})$ 
for all $0 \le j < i \le k$.

For $M^{i,j}$ with $0 \le i \le j \le k-1$, we use the isomorphism 
$M^{i,j} \cong M^{k-i, j-i}$ \eqref{eq:top-weight-Mij} of $M^0$-modules. 
Since $0 \le j-i < k-i \le k$, 
we apply the above identification to $M^{k-i, j-i}$. Then we have 
\begin{equation*}
M^{k-i,j-i} = \M(\Lambda_{j-i} + \Lambda_{(j-i) - (k-i)})=\M(\Lambda_j + \Lambda_{j-i}).
\end{equation*}
Thus the identification $M^{i,j} = \M(\Lambda_j + \Lambda_{j-i})$ holds 
for all $0 \le i \le k$, $0 \le j \le k-1$ in the case (1) of Proposition \ref{prop:two-cases}.

We have discussed the identification of simple modules for $K(\mathfrak{sl}_2,k)$ 
and those for $\W_{k+1,k+2}(\mf{sl}_k)$ 
in the case (1) of Proposition \ref{prop:two-cases} so far. 
As to the case (2) of Proposition \ref{prop:two-cases}, 
recall the permutation $M^{i,j} \mapsto M^{i,j} \circ \theta = M^{i,i-j}$ 
\eqref{eq:theta-action} on the simple 
modules induced by the automorphism $\theta$. 
In the case (1) of Proposition \ref{prop:two-cases}, $M^{i,i-j} = \M(\Lambda_{-j} + \Lambda_{i-j})$. 
Therefore  $M^{i,j} = \M(\Lambda_{-j} + \Lambda_{i-j})$ for all $i,j$ in 
the case (2) of Proposition \ref{prop:two-cases}. 
In fact, $M^p$ and $M^{i,0}$ are transformed to $M^{k-p}$ and $M^{i,i}$ by the 
permutation and the fusion rule \eqref{eq:fusion-Mp-Mij} is transformed to 
$M^{k-p} \times M^{i,i-j} = M^{i,i-(j+p)}$.

We have proved the following theorem.
\begin{thm}\label{thm:irred-mod-correspondence}
There are exactly two ways of identification of the simple modules for $K(\mathfrak{sl}_2,k)$ 
and those for $\W_{k+1,k+2}(\mathfrak{sl}_k)$, 
namely, 

$(1)$ $M^{i,j} = \M(\Lambda_j + \Lambda_{j-i})$ for all $0 \le i \le k$, $0 \le j \le k-1$.

$(2)$ $M^{i,j} = \M(\Lambda_{-j} + \Lambda_{i-j})$ for all $0 \le i \le k$, $0 \le j \le k-1$. 
\end{thm}

The following corollary is already discussed in the proof of 
Theorem \ref{thm:irred-mod-correspondence}.

\begin{cor}\label{cor:theta-on-MLambdaij}
The automorphism $\theta$ of $K(\mathfrak{sl}_2,k) = \W_{k+1,k+2}(\mathfrak{sl}_k)$ induces a permutation 
$\M(\Lambda_j + \Lambda_i) \mapsto \M(\Lambda_{-j} + \Lambda_{-i})$ 
on the simple modules for all $i,j$.
\end{cor}

\appendix
\section{Proof of Proposition \ref{Pro:generator-of-W} }\label{appen}
For a $\Z_{\ge 0}$-graded vertex operator algebra
$V$,
let
$F^p V$ be the subspace
of $V$ spanned by the vectors
$$a^1_{-n_1-1}\cdots a^r_{-n_r-1}b$$
with $a_i\in V$,
$b\in V$,
$n_i\in\Z_{\geq 0}$,
$n_1+\cdots+n_r\geq p$.
We have \cite{Li05}
\begin{align*}
&V=F^0 V\supset F^1 V\supset \cdots,\quad
T F^p V\subset F^{p+1}V,
\quad \bigcap F^p V=0,\\
&a_{n}F^p V\subset F^{p+q-n-1}V
\text{ for $a\in  F^pV$, $n\in \Z$},
\\
&a_{n}F^p V\subset F^{p+q-n}V\quad
\text{for $a\in  F^pV$, $n\geq 0$.}
\end{align*}
Set
$\gr V=\bigoplus_{p=0}^{\infty} F^p V/F^{p+1}V$.
Let
$\sigma_p: F^p V\rightarrow F^pV/F^{p+1}V$
be  the symbol map.
This induces an linear isomorphism
$ V\isomap \gr V$.

\begin{prop}\label{Pro:VPA}
Let $r$ be a non-negative integer such that
\begin{align}
a_{n}F^p V\subset F^{p+q-n+r}V\quad\text{for all $a\in  F^pV$, $n\geq 0$.}
\label{eq:condition}
\end{align}
Then 
$\gr V$ is a  Poisson  vertex algebra by
\begin{align}
 \sigma_p(a)\sigma_q(b)=\sigma_{p+q}(a_{-1}b) ,
\quad T\sigma_p(a)=\sigma_{p+1}(a),
\nonumber
\\
\sigma_p(a)_{(n)}\sigma_q(b)=\sigma_{p+q-n+r}(a_{n}b)\quad \text{for
 }n\geq 0.\label{eq:vp-product}
\end{align}
 \end{prop}
 \begin{proof}The assertion was proved by Li \cite{Li05} for $r=0$.
The same proof applies to  the cases that  $r>0$ as well.
 \end{proof}

Note that in Proposition \ref{Pro:VPA}
we can always take $r=0$,
and this Poisson vertex  algebra
structure of $\gr V$ is trivial  if and only if
$a_{n}F^q \W^\ell(\fing)\subset F^{p+q-n+1}\W^\ell(\fing)$
for all $a \in F^p\W^\ell(\fing)$, $n \geq 0$.
Therefore, if this is the case we can give $\gr V$ a 
Poisson 
vertex  algebra
structure using Proposition \ref{Pro:VPA} for $r=1$.

\begin{exm}\label{eq:Heisenberg} 
The Poisson  vertex algebra $\gr V_{\wg}(\ell,0)$ for $r=0$
in Proposition \ref{Pro:VPA}
is isomorphic
to $\C[J_{\infty}\fing^*]$
equipped with the level $0$ Poisson vertex algebra
structure induced from the 
Kirillov-Kostant  Poisson structure of $\fing^*$
(\cite{Ara12}.)
Here $J_{\infty}X$ denotes the arc space
of a scheme $X$.
Let  $M_{\wh}(\ell,0) \subset V_{\wg}(\ell,0)$ be the Heisenberg vertex subalgebra
generated by $h(-1)\mathbf{1}$ with $h\in \finh$.
Then $\gr M_{\wh}(\ell,0) \cong \C[J_{\infty}\h^*]$.
Here,  again, $\C[J_{\infty}\h^*]$
equipped with the level $0$ Poisson vertex algebra
structure induced from the
Kirillov-Kostant  Poisson structure of $\finh^*$,
which is trivial.
Therefore, \eqref{eq:condition} holds for $r=1$.
Hence \eqref{eq:vp-product}
for $r=1$ gives the Poisson vertex  algebra
structure on $\gr M_{\wh}(\ell,0)$.
We have
\begin{align*}
h_{n}h'=\begin{cases}
	    \ell (h|h') &\text{for }n=1\\
0&\text{for }n=0\text{ or }n\geq 2,
	   \end{cases}
\end{align*}
for $h,h'\in \finh\subset \C[\finh^*]\subset \C[J_{\infty}\finh^*]$.
In particular the Poisson vertex algebra structure of 
$\gr M_{\wh}(\ell,0)$ does not depend on the level  $\ell$ provided that $\ell \ne 0$.
 \end{exm}

Consider the $W$-algebra $\W^{\ell}(\g)$.
The vertex Poisson  algebra
structure of $\gr \W^{\ell}(\fing)$ is trivial 
if we take $r$ to be $0$
in Proposition \ref{Pro:VPA} (\cite{Ara10}).
Therefore, by Proposition \ref{Pro:VPA}
$\gr \W^{\ell}(\fing)$ is the vertex Poisson algebra by 
$\sigma_p(a)_{(n)}\sigma_q(b)=\sigma_{p+q-n+1}(a_{n}b)$,
$n\geq 0$.
Below we regard 
$\gr \W^{\ell}(\fing)$ 
as a vertex Poisson algebra
with respect to this product.

In order to prove Proposition \ref{Pro:generator-of-W}
it is sufficient to show the following propositions.

\begin{prop}\label{Pro:generator-of-W-Poisson}
Let $\ell$ be non-critical.
Then  the vertex Poisson algebra
$\gr \W^{\ell}(\mf{sl}_k)$ is  generated by the 
weight $2$ subspace
and the weight $3$ subspace.
\end{prop}
 
\begin{prop}\label{Pro:independence-of-VPA-structure}
For a non-critical $\ell$, the vertex Poisson algebra structure of 
$\gr \W^{\ell}(\g)$
is independent of $\ell \in \C\backslash \{-h^{\vee}\}$.
\end{prop}

\begin{proof}
We shall use the Miura map
$$\W^{\ell}(\g) \hookrightarrow M_{\wh}(\ell+h^{\vee},0),$$
see \cite{APisa} for the details.
This induces the injective Poisson vertex algebra homomorphism
$$\gr \W^{\ell}(\g) \hookrightarrow \gr M_{\wh}(\ell+h^{\vee},0),$$
where $\gr M_{\wh}(\ell+h^{\vee},0)$ is equipped with 
the Poisson vertex algebra structure described in 
Example \ref{eq:Heisenberg}.
The image of $\gr \W^{\ell}(\g)$ is generated by symmetric polynomials
in $S(\h)=\C[\h^*]\subset \C[J_{\infty}\h^*]$, 
and does not depend on $\ell$.
Hence, the vertex Poisson algebra structure of 
$\gr \W^{\ell}(\g)$
is independent of $\ell$ as long as it is non-critical.
\end{proof}

By Proposition \ref{Pro:independence-of-VPA-structure},
it is sufficient to show Proposition \ref{Pro:generator-of-W-Poisson}
for  a non-critical $\ell$.

Recall the following assertion proved by Frenkel, Kac, Radul and Wang.
\begin{thm}[\cite{FKRW}]
For $\ell=1-k$,
 $\W^{\ell}(\mf{gl}_k) $ is isomorphic to
the simple quotient of $\W_{1+\infty}$-algebra
$\W_{1+\infty}^k$
of central charge $k$.
\end{thm}

\begin{proof}[Proof of Proposition \ref{Pro:generator-of-W-Poisson}]
The vertex algebra
$\W_{1+\infty}^c$ is freely generated by fields
$J^m(z)=\sum_{n \in \Z}J^m_{n}z^{-n-1}$, $m=0,1,2,\dots$,
satisfying the OPE's
\begin{align*}
& J^m(z) J^n(w) \\
&\sim \sum_{a=1}^{m+n}([n]_a
J^{m+n-a}(w)-(-1)^a[m]_aJ^{m+n-a}(z)/(z-w)^{a+1})
+
 \frac{(-1)^m m!n! c}{(z-w)^{m+n+2}},
\end{align*}
where
$[n]_a=n(n-1)\cdots (n-a+1)$.
The conformal weight of $J^m(z)$
is $m+1$.
The image of $J^0(z)$  generates 
the rank $1$ Heisenberg subalgebra $\pi$ 
and we have
$\W^{\ell}(\mf{gl}_k)=\W^{\ell}(\mf{sl}_k) \otimes \pi$.

We have
\begin{align*}
& J^m_{0}J^n\equiv 0
\pmod{F^1\W_{1+\infty}^c}, \quad 
J^m_{1}J^n\equiv (m+n)J^{m+n-1}
\pmod{F^1 \W_{1+\infty}^c},
\\& J^m_{r}J^n=([n]_r-(-1)^r[m]_r)J^{m+n-r}\pmod{F^1 \W^c_{1+\infty}}
\end{align*}

It follows that
$\gr \W_{1+\infty}^c$ is generated by 
$J^0$, $J^1$ and $J^2$.
Therefore 
$\gr \W^{1-k}(\mf{sl}_k)$
is generated by the image of $J^1$ and $J^2$.
This completes the proof.
\end{proof}

\section*{Acknowledgments}
The authors would like to thank Toshiyuki Abe, Scott Carnahan, Chongying Dong, 
Atsushi Matsuo and Hiroshi Yamauchi for helpful advice and stimulating discussions. 
The authors learn Lemma 2.1 from Toshiyuki Abe. 
Part of the work was done while C. L. and H. Y. were staying
at Kavli Institute for Theoretical Physics China, Beijing in July and August, 2010, 
T. A, C. L. and H. Y. were staying at National Center for Theoretical Sciences (South), 
Tainan in September, 2010, 
and T. A. and H. Y. were staying at Academia Sinica, Taipei in December, 2011. 
They are grateful to those institutes. 
Tomoyuki Arakawa was partially supported by the JSPS Grant-in-Aid for
Scientific Research No. 25287004 and No. 26610006, 
Ching Hung Lam was partially supported by 
MoST grant 104-2115-M-001-004-MY3 of Taiwan, 
Hiromichi Yamada was partially supported by JSPS Grant-in-Aid for
Scientific Research No. 26400040.

\bibliographystyle{alpha}
\bibliography{math}


\end{document}